 \documentclass[anona]{classe}    

\usepackage{mathrsfs}
\DeclareMathOperator*{\lOngrightarrow}{\longrightarrow}      
\usepackage{dsfont}
\newcommand{\be} {\begin{equation}}
\newcommand{\ee} {\end{equation}}
\newcommand{\bea} {\begin{eqnarray}}
\newcommand{\eea} {\end{eqnarray}}
\newcommand{\RR}{{\mathbb R}}
\newcommand{\NN}{{\mathbb N}}
\newcommand{\eqdef}{\stackrel{{\rm {def}}}{=}}
\title{Quasilinear and singular elliptic systems}
\headlinetitle{}

\lastnameone{Giacomoni}
\firstnameone{Jacques}
\nameshortone{J.Giacomoni}
\addressone{LMAP - UMR CNRS 5142 
B\^atiment IPRA - 
Avenue de l'universit\'e - BP 1155, F- 64013 Pau}
\countryone{France}
\emailone{jacques.giacomoni@univ-pau.fr}

\lastnametwo{Hernandez}
\firstnametwo{Jes\'us}
\nameshorttwo{J. Hernandez}
\addresstwo{Departamento de Matem\'aticas. Universidad Aut\'onoma de Madrid\\
28049 Madrid}
\countrytwo{Spain}
\emailtwo{jesus.hernandez@uam.es}

\lastnamethree{Sauvy}
\firstnamethree{Paul}
\nameshortthree{P. Sauvy}
\addressthree{LMAP - UMR CNRS 5142 
B\^atiment IPRA -
Avenue de l'universit\'e - BP 1155, F- 64013 Pau}
\countrythree{France}
\emailthree{paul.sauvy@etud.univ-pau.fr}

\abstract{In this paper, we investigate the following quasilinear elliptic and singular system $\boldsymbol{({\rm P})}$: 
$$-\Delta_pu=f_1(x,u,v)\quad\text{in }\Omega\, ;\quad u|_{\partial\Omega}=0,\quad u>0\quad\text{in }\Omega,$$
$$-\Delta_qv=f_2(x,u,v)\quad\text{in }\Omega\, ;\quad v|_{\partial\Omega}=0,\quad v>0\quad\text{in }\Omega,$$
where $\Omega$ is a bounded domain with smooth boundary in $\RR^{N}$, $1<p,q<\infty$ and $f_1,f_2\in \mathscr{C}^1\left(\Omega\times\RR^{\ast}_+\times\RR^{\ast}_+\right)$ two positive functions. Under suitable conditions on $f_1$ and $f_2$, we first give a general result on the existence of positive weak solutions pairs $(u,v)\in\mathrm{W}^{1,p}_0(\Omega)\times\mathrm{W}^{1,q}_0(\Omega)$ to $\boldsymbol{({\rm P})}$. Next, we give some applications to Biology. }



\begin{document}

\section{Introduction}

In this paper we are interested in the following quasilinear elliptic and singular system, 
$$ \boldsymbol{({\rm P})}\left\lbrace
\begin{array}{ll}
          -\Delta_pu=f_1(x,u,v)\quad\text{in }\Omega\, ;\quad u|_{\partial\Omega}=0,\quad u>0\quad\text{in }\Omega, \vspace{0.2cm} \\
          
          -\Delta_qv=f_2(x,u,v)\quad\text{in }\Omega\, ;\quad v|_{\partial\Omega}=0,\quad v>0\quad\text{in }\Omega.
 \end{array}
\right.
$$
Here, $\Omega$ is a bounded domain of $\RR^N$, $N\geq2$ with  $\mathscr{C}^2$ boundary $\partial \Omega$, $\Delta_ru\eqdef\mathrm{div}(\vert\nabla u\vert^{r-2}\nabla u)$ denotes the $r\,$\--Laplace operator and $1<p,q<\infty$. In the right-hand sides, $f_1$ and $f_2$ are two Carath\'eodory functions in $\Omega\times\left(\RR^{\ast}_+\times\RR^{\ast}_+\right)$  possibly singular. More precisely, for every $(t_1,t_2)\in \RR^{\ast}_+\times\RR^{\ast}_+$ and for almost every $x\in \Omega$, { we assume that}

  $\boldsymbol{(H_1)}\qquad$ $f_1(\cdot,t_1,t_2)$ and $ f_2(\cdot,t_1,t_2)$ are Lebesgue measurable in $\Omega,$
  
  $\boldsymbol{(H_2)}\qquad$ $ f_1(x,\cdot,\cdot)$ and $f_2(x,\cdot,\cdot)$ are in $\mathscr{C}^1(\RR^{\ast}_+\times\RR^{\ast}_+).$\\
We aim {to establish} the existence of a positive weak solutions pair to problem $\boldsymbol{({\rm P})}$ using the Schauder Fixed Point {Theorem}. Namely, if we can compose two order-reversing mappings, 
\be\label{T}(u,v)\mapsto T_1(u,v)\eqdef \tilde{u}\quad\text{ and } \quad (u,v)\mapsto T_2(u,v)\eqdef\tilde{v},\ee where $\tilde{u}\in{\rm W}^{1,p}_0(\Omega)$ and $\tilde{v}\in{\rm W}^{1,q}_0(\Omega)$ are defined to be the (unique) positive weak solution to the Dirichlet problems
\be\label{T1}-\Delta_p \,\tilde{u}+h_1(x,\tilde{u})=f_1\left(x,{u},v\right)+h_1(x,u)\;\text{in }\Omega;\;\; \tilde{u}|_{\partial\Omega=0},\quad\tilde{u}>0\quad \text{in }\Omega,\ee
\be\label{T2}-\Delta_q \,\tilde{v}+h_2(x,\tilde{v})=f_2\left(x,u,{v}\right)+h_2(x,v)\;\text{in }\Omega;\;\;\tilde{v}|_{\partial\Omega=0},\quad\tilde{v}>0\quad \text{in }\Omega,\ee
respectively, in suitable conical shells of positive cones in ${\rm W}^{1,p}_0(\Omega)$ and ${\rm W}^{1,q}_0(\Omega)$, with appropriate functions $h_1$ and $h_2$; then any fixed point of the mapping 
\be\label{Tbis}(u,v)\mapsto{T}(u,v)\eqdef(T_1(u,v),T_2(u,v))\ee
is a positive weak solution pair to $\boldsymbol{({\rm P})}$ and conversely. To prove that $T$ is well defined and invariant in some conical shell, we use monotonicity methods together with the existence of sub\-- and supersolutions which prescribe the behaviour of the right\--hand side singular non-linearities, namely $f_1$ and $f_2$, near the boundary $\partial\Omega$. The continuity and the compactness in $\mathscr{C}^{0,\alpha}(\overline{\Omega})\times \mathscr{C}^{0,\alpha}(\overline{\Omega})$ for some suitable $0<\alpha<1$ follow from the regularity result Theorem 1.1 in \cite{GST1} we recall in the appendix (see Theorem \ref{regu}). We derive further uniqueness results in case where the system $\boldsymbol{({\rm P})}$ is competitive or cooperative (see Theorem \ref{uniq}). To establish the uniqueness of a positive pair of solutions to $\boldsymbol{({\rm P})}$, it is essential that the mapping $T$ is {\em subhomogeneous\/}. In the cooperative and "strong" singular case, we also prove the existence of very weak solutions in ${\rm W}^{1,p}_{{\rm loc}}(\Omega)\times {\rm W}^{1,q}_{{\rm loc}}(\Omega)$ (see Theorem \ref{TH2}).

Quasilinear elliptic systems have been quite intensely investigated in the literature with various methods. In \cite{ThVe}, the authors take advantage of the variational structure of the problem to apply variational methods. In \cite{ClMaMi}, a blow up argument combined with a Liouville theorem yields universal a priori bounds. Then, the existence of solutions is obtained by a topological degree argument (see also the review article \cite{Defig}). In \cite{CuTa}, the key ingredients to prove existence of solutions are the Strong Comparison Principle and Kre{\v\i}n-Rutman theorem for homogeneous non-linear mapping. While dealing with subhomogeneous systems, one usually appeals the method of sub and supersolutions.

Related problems for \textit{singular} quasilinear systems have been also studied in
\cite{LeShYe} and \cite{GHM}. Accordingly, we study in our paper a more general situation that handle more singular cases. We point out additionally that in the present work non-linearities $f_1$ and $f_2$ are not necessary non-negative.

The case of singular semi-linear systems ($p=q=2$) has been studied
even more frequently in
\cite{ChKe}, \cite{ChKe2}, \cite{Ni}, \cite{Her}, \cite{HeMa} and \cite{Gh}.
We refer to \cite{Her} for additional references on the subject.

Throughout this paper, we will use the following notations and definitions:
\begin{enumerate}
\item To $r\in(1,+\infty)$ we associate $r'\eqdef\frac{r}{r-1}>1$ and we denote by ${\rm W}^{-1,r'}(\Omega)$ the dual space of ${\rm W}^{1,r}_0(\Omega)$ with respect to the standard inner product in ${\rm L}^2(\Omega)$.
\item We denote by $d(x)\eqdef   \inf\limits_{y\in\partial\Omega}d(x,y)$, the distance from $x\in\Omega$ to $\partial\Omega$.
\item We denote by $D\eqdef \sup\limits_{x,y\in\Omega}d(x,y)$, the diameter of the domain $\Omega$.
\item Let $f,g: \Omega\longrightarrow [0,+\infty]$ be two functions of ${\rm L}^{1}_{{\rm loc}}(\Omega).$ Then, we write $$f(x) \,\sim\, g(x)\qquad\text{ in }\Omega$$ if there exist two positive constants $C_1$ and $C_2$ such that for almost every $x\in\Omega$,
\be\nonumber\label{def1} C_1g(x)\leq f(x)\leq C_2g(x).\nonumber\ee
\item The function $\varphi_{1,r}\in \mathrm{W}^{1,r}_0(\Omega)$ denotes the positive and $\mathrm{L}^r$-renormalized 
 eigenfunction corresponding to the first eigenvalue of $-\Delta_r$,
\begin{small}$$\lambda_{1,r}\eqdef\inf\left\lbrace\int_{\Omega}\vert\nabla v\vert^rdx\in\RR,\quad v\in \mathrm{W}^{1,r}_0(\Omega)\quad\text{and}\quad\int_{\Omega}\vert v\vert^rdx=1\right\rbrace.$$                                                                                                                                               \end{small}
It is a weak solution of the following eigenvalue problem:
$$
-\Delta_r w = \lambda_{1,r} w^{r-1}\text{ in }\Omega;\quad w|_{\partial\Omega}=0,\quad w> 0 \quad\text{ in }\Omega.
$$
Using Moser iterations, $\varphi_{1,r}\in \mathrm{L}^{\infty}(\Omega)$ and using the H\" older regularity result in \textsc{Lieberman} \cite{Lie}, $\varphi_{1,r}\in\mathscr{C}^{1,\alpha}\left(\overline{\Omega}\right)$ for some $0<\alpha<1$.
Moreover the strong maximum and boundary principles from \textsc{V\'asquez} \cite{Vas}, guarantee that $\varphi_{1,r}$ satisfies 
\be\label{vep} \varphi_{1,r}(x)\sim d(x)\quad\text{ in }\Omega.\ee
\item We say that a Lebesgue measurable function $f:\Omega\rightarrow\RR$ is \textit{locally uniformly positive} if $\mathrm{essinf}_{K}f>0$ holds over every compact set $K\subset\Omega$. 
\item In this paper, we primarily look for \textit{{positive weak solution pairs}} (\textit{{positive solutions}}, for short) of problem $\boldsymbol{({\rm P})}$, that is, pairs of functions $(u,v)\in{\rm W}^{1,p}_0(\Omega)\times{\rm W}^{1,q}_0(\Omega)$ with both $u$ and $v$ locally uniformly positive and each satisfying the respective equation in problem $\boldsymbol{({\rm P})}$ in the weak sense. More precisely, given $1<r<\infty$ and $f\in{\rm W}^{-1,r'}(\Omega)$, we say that a function $u\in{\rm W}^{1,r}_0(\Omega)$ satisfies the equation 
\be\label{bo} -\Delta_ru=f\quad \text{ in }\Omega\ee
in the weak sense if $u $ is locally uniformly positive and satisfies
$$\forall w\in{\rm W}^{1,r}_0(\Omega),\quad\int_{\Omega}\vert\nabla u\vert^{r-2}\nabla u.\nabla w\,dx=\langle f,w\rangle_{{\rm W}^{-1,r'}(\Omega)\times {\rm W}^{1,r}_0(\Omega)}.$$
In the case where the existence of positive solutions of $\boldsymbol{({\rm P})}$ cannot be established, we discuss the existence of weaker solutions. Then, we say that $(u,v)\in{\rm W}^{1,p}_{\rm loc}(\Omega)\times{\rm W}^{1,q}_{\rm loc} (\Omega)$ is a \textit{positive very weak solution pair} of $\boldsymbol{({\rm P})}$ if both $u$ and $v$ are locally uniformly positive and satisfy the respective equation in problem $\boldsymbol{({\rm P})}$ in the sense of distributions. \\

 In the three last points, for $1<r<+\infty$, $\mathscr{A}_r(\Omega)$ represents the space ${\rm W}^{1,r}_0(\Omega)$ or the space ${\rm W}^{1,r}_{\rm loc}(\Omega)$.\\
 
 \item Let $\underline{w},\overline{w}\in\mathscr{A}_r(\Omega)$, two locally uniformly positive functions such that $\underline{w}\leq\overline{w}$ a.e. in $\Omega$. We define the convex set $$\left[\underline{w},\overline{w}\right]\eqdef\left\lbrace w\in\mathscr{A}_r(\Omega)\cap\mathscr{C}\left(\overline{\Omega}\right),\quad \underline{w}\leq w\leq\overline{w}\quad \text{ a.e. in }\Omega\right\rbrace.$$
\item Let $\underline{u},\overline{u}\in\mathscr{A}_p(\Omega)$ and $\underline{v},\overline{v}\in\mathscr{A}_q(\Omega)$ four locally uniformly positive functions such that $\underline{u}\leq\overline{u}$ a.e. in $\Omega$ and  $\underline{v}\leq\overline{v}$ a.e. in $\Omega$. The couples $(\underline{u},\underline{v})$ and $(\overline{u},\overline{v})$ are said to be \textit{sub and supersolutions} pairs to $\boldsymbol{({\rm P})}$ if the following inequalities are satisfied in the distribution sense
\bea
\label{soussol}-\Delta_p\,\underline{u}\leq f_1\left(x,\underline{u},v\right)\quad \text{ in }\Omega,\quad \text{ for any }v\in\left[\underline{v},\overline{v}\right],\vspace{0.2cm}\\
-\Delta_q\,\underline{v}\leq f_2\left(x,u,\underline{v}\right)\quad \text{ in }\Omega,\quad \text{ for any }u\in\left[\underline{u},\overline{u}\right],\vspace{0.2cm}\\
\label{sursolbisbis}-\Delta_p\,\overline{u}\geq f_1\left(x,\overline{u},v\right)\quad \text{ in }\Omega,\quad \text{ for any }v\in\left[\underline{v},\overline{v}\right],\vspace{0.2cm}\\
\label{sursol}-\Delta_q\,\overline{v}\geq f_2\left(x,u,\overline{v}\right)\quad \text{ in }\Omega,\quad \text{ for any }u\in\left[\underline{u},\overline{u}\right].
\eea
\item Let$(\underline{u},\underline{v})$, $(\overline{u},\overline{v})\in\mathscr{A}_p(\Omega)\times\mathscr{A}_q(\Omega)$ be  respectively sub and supersolutions pairs to $\boldsymbol{({\rm P})}$. Then, the conical shell $[\underline{u},\overline{u}]\times[\underline{v},\overline{v}]$ is denoted by $\mathcal{C}$.
\end{enumerate}
The paper is organised as follows. The next section (Section 2) contains the statements and the proofs of our main results (Theorems \ref{TH1} and Theorems \ref{TH2}). Different applications of Theorems \ref{TH1} and \ref{TH2} arising in population dynamics models are given in Section 3. The appendix contains the regularity result (Theorem \ref{regu}) used to prove H\"older continuity of solutions. Theorem \ref{regu} is proved in \cite{GST1}.
\section{  General results }
\begin{theorem}\label{TH1}
Let $(\underline{u},\underline{v}),(\overline{u},\overline{v})\in \mathrm{W}^{1,p}_0(\Omega)\times {\rm W}^{1,q}_0(\Omega)$ be sub and supersolutions pairs to $\boldsymbol{({\rm P})}$ and assume in addition that the following conditions hold:
\begin{enumerate}
\item there exist constants $k_1, k_2>0$ and $\delta_1,\delta_2\in\RR$ such that 
\be\label{f1}
 \vert f_1(x,u,v)\vert\leq k_1 d(x)^{\delta_1} \quad\text{and}\quad \vert  f_2(x,u,v)\vert\leq k_2 d(x)^{\delta_2} \quad\text{in }\;\Omega\times\mathcal{C},\ee
\item there exist constants $C_1,C_2>0$ and $b_1,b_2>0$ such that
\be\label{ineqsursol} 
 \overline{u}\leq C_1 d(x)^{b_1}\quad\text{ and }\quad \overline{v} \leq C_2 d(x)^{b_2}\quad\text{in }\;\Omega,\\
\ee
\item and there exist $\kappa_1,\kappa_2>0$ and $\alpha_1, \alpha_2>0$ such that
\be\label{df1}
 \left| \frac{\partial f_1}{\partial u}(x,u,v)\right|\leq \kappa_1 d(x)^{\delta_1-\alpha_1}\quad\text{in}\; \Omega\times\mathcal{C},\ee
\be\label{df2}
  \left| \frac{\partial f_2}{\partial v}(x,u,v)\right|\leq \kappa_2 d(x)^{\delta_2-\alpha_2}\quad\text{in}\; \Omega\times\mathcal{C},
\ee
\end{enumerate}
with the following conditions on the coefficients 
\be\label{delta1}
	\delta_1>-2+\frac{1}{p}+(\alpha_1-b_1)^+,
\qquad
	\delta_2>-2+\frac{1}{q}+(\alpha_2-b_2)^+.	
\ee
Then, there exists a  positive weak solutions pair $(u,v)\in \mathcal{C}$. 
\end{theorem}
\begin{remark}
 Instead of conditions \eqref{df1} and \eqref{df2}, as in \cite{GHM}, we { can rather suppose} that there exist $\kappa_1,\kappa_2>0$ and $\alpha_1, \alpha_2>0$ such that for all $(u,v)\in \mathcal{C}$,
 $$w\mapsto f_1(x,w,v) +\kappa_1 d(x)^{\delta_1-\alpha_1}w^{p-1}\text{ is non decreasing on } [\underline{u},\overline{u}], $$ 
 $$w\mapsto f_2(x,u,w) +\kappa_2 d(x)^{\delta_2-\alpha_2}w^{q-1}\text{ is non decreasing on } [\underline{v},\overline{v}].$$  Replacing condition \eqref{delta1} by
\be \delta_1>-2-\frac{1}{p}+(\alpha_1-(p-1) b_1)^+,
\qquad
\delta_2>-2+\frac{1}{q}+\left(\alpha_2-\left(q-1\right)b_2\right)^+,\nonumber
\ee
we get the same result and the condition is sharper if $p,q>2$. For that, it suffices to replace the first equation of the problem $\boldsymbol{({\rm Q})}$,  given below, by
$$
-\Delta_{p}w +\tilde{g}_1(x,w)=f_1(x,u,v)+\kappa_1d(x)^{\delta_1-\alpha_1} u^{p-1}\;\text{ in }\,\Omega,\\
$$
with $\tilde{g}_1:\Omega\times \RR\rightarrow \RR^{\ast}_+$ the cut-off function defined as follows:
\be \tilde{g}_1(x,z)\eqdef\left\lbrace\begin{array}{cl}
\kappa_1d(x)^{\delta_1-\alpha_1}\overline{u}^{p-1}&\quad\text{if } z\geq\overline{u}(x),\\
\kappa_1d(x)^{\delta_1-\alpha_1}z^{p-1} & \quad\text{if }z\in\left[0,\overline{u}(x)\right],\\
0&\quad \text{if }z\leq 0
\end{array}
\right. 
\ee 
and proceed similarly for the second equation of $\boldsymbol{({\rm P})}$.
\end{remark}

\begin{proof}
Let $(u,v)\in\mathcal{C}$. We first prove the existence of $T_1(u,v)\in{\rm W}^{1,p}_0(\Omega)$, where $T_1(u,v)$ is defined in \eqref{T1} with $h_1(x,u)\eqdef\kappa_1 d(x)^{\delta_1-\alpha_1}u$ in $\Omega\times[\underline{u},\overline{u}]$.
For that, let us introduce the following problem :
$$\boldsymbol{({\rm Q})}
\left\lbrace
\begin{array}{l}
 -\Delta_{p}w +g_1(x,w)=f_1(x,u,v)+\kappa_1d(x)^{\delta_1-\alpha_1} u\;\text{ in }\,\Omega,\vspace{0.2cm}\\
  \;\;w\vert_{\partial\Omega}=0,\quad w>0\quad\text{in }\Omega,\\
\end{array}
\right.
$$ 
with $g_1:\Omega\times \RR\rightarrow \RR^{\ast}_+$ the cut-off function defined as follows:
\be g_1(x,z)\eqdef\left\lbrace\begin{array}{cl}
\kappa_1d(x)^{\delta_1-\alpha_1}\overline{u}&\quad\text{if } z\geq\overline{u}(x),\\
\kappa_1d(x)^{\delta_1-\alpha_1}z & \quad\text{if }z\in\left[0,\overline{u}(x)\right],\\
0&\quad \text{if }z\leq 0.
\end{array}
\right. 
\ee 
Then, $g_1$ is a Carath\'eodory function on $\Omega\times\RR$. Thus, for $(x,s)\in\Omega\times\RR$, setting $\displaystyle G_1(x,s)\eqdef\int_0^sg_1(x,z)dz$, we consider the following functional: $\forall w\in{\rm W}^{1,p}_0(\Omega),$
{\small $$E(w)\eqdef\frac{1}{p}\int_{\Omega}\vert \nabla w\vert^p\,dx+\int_{\Omega}G_1(x,w)\,dx-\int_{\Omega}\left(f_1(x,u,v)+\kappa_1d(x)^{\delta_1-\alpha_1}u\right)w\,dx.$$}
By assumption \eqref{delta1} and Hardy's inequality, $E$ is well defined in ${\rm W}^{1,p}_0(\Omega)$ and for all $w\in{\rm W}^{1,p}_0(\Omega)$,
{\small \be\label{fonc} E(w)\geq \\\frac{1}{p}\Vert w\Vert^p_{{\rm W}^{1,p}_0(\Omega)}-C\left\Vert \left(f_1(x,{u},v)+\kappa_1 d(x)^{\delta_1-\alpha_1}u\right)d(x)\right\Vert_{{\rm L}^{p'}(\Omega)}\Vert w\Vert_{{\rm W}^{1,p}_0(\Omega)}.\ee} So, let us define 
\be I\eqdef \inf\limits_{w\in{\rm W}^{1,p}_0(\Omega)}E(w)\ee
and let $\left(w_n\right)_{n\in\NN}\subset {\rm W}^{1,p}_0(\Omega)$
 be a minimizing sequence of $E$, \textit{i.e. }$\lim\limits_{n\to\infty} E(w_n)=I$. Using \eqref{fonc}, $\left(w_n\right)_{n\in\NN}$ is bounded in ${\rm W}^{1,p}_0(\Omega)$, therefore there exists a subsequence $\left(w_{n_k}\right)_{k\in \NN}$ and $\tilde{u}\in {\rm W}^{1,p}_0(\Omega)$ such that $w_{n_k}\lOngrightarrow\limits_{k\to\infty}\tilde{u}$, weakly in ${\rm W}^{1,p}_0(\Omega)$ and a.e. in $\Omega$. Therefore,
$$ \liminf\limits_{k\to\infty}\Vert w_{n_k}\Vert_{{\rm W}^{1,p}_0(\Omega)}\geq\Vert \tilde{u}\Vert_{{\rm W}^{1,p}_0(\Omega)}$$
and using Fatou's lemma, 
 $$\liminf\limits_{k\to \infty}\int_{\Omega}G_1(x,w_{n_k}) dx \geq \int_{\Omega}\liminf\limits_{k\to\infty} G_1(x,w_{n_k})\,dx=\int_{\Omega}G_1(x,\tilde{u}) dx.$$ Hence, $E(\tilde{u})=I$ and $\tilde{u}$ is a solution to the Euler-Lagrange equation associated to $E_0$, { that is}: 
 \begin{multline}\label{EL}\quad \int_{\Omega}\vert \nabla \tilde{u}\vert^{p-2}\nabla \tilde{u}.\nabla w\,dx+\int_{\Omega}g_1(x,\tilde{u})w\,dx\\
=\int_{\Omega}\left(f_1(x,u,v)+\kappa_1 d(x)^{\delta_1-\alpha_1}u\right) w\,dx,\end{multline}
for any $ w\in{\rm W}^{1,p}_0(\Omega).$ Now let us prove that $\tilde{u}\in[\underline{u},\overline{u}]$. Combining \eqref{soussol} and \eqref{EL}, we get for all $ w\in {\rm W}^{1,p}_0(\Omega)^+\eqdef\lbrace{w\in{\rm W}^{1,p}_0(\Omega),\; w\geq 0\; \text{ a.e in }\Omega\rbrace}$,
\begin{small}
\begin{multline}
\int_{\Omega}\left(\vert \nabla \tilde{u}\vert^{p-2}\nabla \tilde{u}-\vert \nabla\underline{u}\vert^{p-2}\nabla \underline{u}\,\right). \nabla w\, dx
+ \int_{\Omega} \left(g_1(x,\tilde{u})-g_1(x,\underline{u})\right)w\,dx\\
\geq \int_{\Omega}\left[ \left(f_1(x,u,v)+\kappa_1d(x)^{\delta_1-\alpha_1}u\right)-\left(f_1(x,\underline{u},v)+\kappa_1d(x)^{\delta_1-\alpha_1}\underline{u}\right)\right]w\, dx.
\end{multline}
\end{small}
By assumption \eqref{df1}, applying this inequality with $w=(\tilde{u}-\underline{u})^-\in{\rm W}^{1,p}_0(\Omega)^+$, we get $\tilde{u}\geq \underline{u}$ a.e. in $\Omega$. Similarly, combining \eqref{sursolbisbis} and \eqref{EL} we also get $\tilde{u}\leq \overline{u}$ a.e. in $\Omega$. Then, $\tilde{u}$ satisfies the equation 
\be\label{T1(u,v)}-\Delta_p \tilde{u}+\kappa_1d(x)^{\delta_1-\alpha_1}\tilde{u}=f_1(x,u,v)+\kappa_1d(x)^{\delta_1-\alpha_1} u \quad\text{ in }\Omega,\ee
in the weak sense. Moreover, using a classical local regularity result in \cite{Ser}, $\tilde{u}\in\mathscr{C}^{1,\gamma}\left({K}\right)$ for some $\gamma>0$ in any compact subset $K$ of $\Omega$. So using inequality \eqref{ineqsursol}, $\tilde{u}\in \mathscr{C}\left(\overline{\Omega}\right)$, which gives us that $\tilde{u}\in[\underline{u},\overline{u}]$. Finally, by the weak maximum principle, $\tilde{u}$ is the unique function in the conical shell $[\underline{u},\overline{u}]$ satisfying \eqref{T1(u,v)}. Then, the mapping $T_1: (u,v)\mapsto\tilde{u}$ is well-defined from $\mathcal{C}$ to $[\underline{u},\overline{u}]$. In the same spirit, we get the existence of the mapping $T_2:(u,v)\mapsto \tilde{v}$ defined from $\mathcal{C}$ to $[\underline{v},\overline{v}]$, where $\tilde{v}$ is the unique weak solution in $[\underline{v},\overline{v}]$ of
\be-\Delta_p \tilde{v}+\kappa_2d(x)^{\delta_2-\alpha_2}\tilde{v}=f_2(x,u,v)+\kappa_2d(x)^{\delta_2-\alpha_2} v\quad\text{ in }\Omega.\ee
This proves that the operator $T$ defined in \eqref{Tbis} is well-defined and makes invariant the conical shell ${\mathcal C}$.
\\
\\
Now, the continuity and the compactness of $T$ follow from a regularity result in \cite{GST1} we recall in appendix A. Indeed, let $(u_n,v_n)_{n\in\NN}\subset\mathcal{C}$ and $({u},{v})\in\mathcal{C}$ such that 
$(u_n,v_n)\rightarrow(u,{v})\text{ in }\mathscr{C}\left(\overline{\Omega}\right)\times\mathscr{C}\left(\overline{\Omega}\right)$ as $n\to+\infty$.
Then, from Theorem \ref{regu} and assumptions \eqref{f1}, $\left(T_1(u_n,v_n)=\tilde{u}_n\right)_{n\in\NN}$  is bounded in $\mathscr{C}^{0,\alpha}(\overline{\Omega})$, for some $0<\alpha<1$. By Ascoli-Arzel\`a theorem, there exists a subsequence $(\tilde{u}_{n_k})_{k\in\NN}$ and $\tilde{u}\in[\underline{u},\overline{u}]$ such that  
$\tilde{u}_{n_k}\rightarrow\tilde{u}$ uniformly in $\overline{\Omega}$ as $k\to\infty$.
Moreover, using the local regularity result in \cite{Ser}, $(\tilde{u}_{n_k})_{k\in\NN}$ is bounded in $\mathscr{C}^{1,\gamma}\left({K}\right)$
for some $\gamma>0$ and for any compact subset $K$ of $\Omega$ which entails that up to a subsequence denoted again $(\tilde{u}_{n_k})_{k\in\NN}$ such that $\nabla \tilde{u}_{n_k}\rightarrow\nabla \tilde{u} $ uniformly  in $K$ as $k\to+\infty$.
Then, $\tilde{u}$ satisfies
\be\label{distributions}-\Delta_p\tilde{u}+\kappa_1d(x)^{\delta_1-\alpha_1}\tilde{u}=f_1(x,u,v)+\kappa_1d(x)^{\delta_1-\alpha_1}u\quad\text{ in }\Omega\ee
in the sense of distributions.
Moreover, since $\tilde{u}\leq \overline{u}$ a.e in $\Omega$, $f_1(x,u,v)+\kappa_1d(x)^{\delta_1-\alpha_1}(u-\tilde{u})\in{\rm W}^{-1,p'}(\Omega)$, which implies that $\tilde{u}\in {\rm W}^{ 1,p}_0(\Omega)$. Hence $\tilde{u}\in[\underline{u},\overline{u}]$ and is a weak solution of \eqref{distributions}. By uniqueness of a such solution in $[\underline{u},\overline{u}]$, it follows that $\tilde{u}=T_1(u,v)$ and all the sequence $(\tilde{u}_n)_{n\in\NN}$ converges to $\tilde{u}$ in $\mathscr{C}(\overline{\Omega})$. The same arguments hold to prove that $T_2(u_n,v_n)\to T_2(u,v)$ uniformly in $\overline{\Omega}$ as $n\to+\infty$. Then, $T:\mathcal{C}\to\mathcal{C}$ is continuous.
Finally, it easy from the compact embedding of $\mathscr{C}^{0,\alpha}(\overline{\Omega})$ in $\mathscr{C}(\overline{\Omega})$ to  get the compactness of $T$. Applying the Schauder Fixed Point Theorem to $T$ in $\mathcal{C}$, the proof  of Theorem \ref{TH1} is now complete.
\end{proof}
We now give a more general result which guarantees the existence of a "very weak" positive solutions pair, in the cooperative case, when the inequalities \eqref{delta1} { may not be satisfied}.
\begin{theorem}\label{TH2}
Assume that $ \boldsymbol{({\rm P})}$ is a \textit{\textbf{cooperative system}}, \textit{i.e.} 
\be\label{coop} \frac{\partial f_1}{\partial v}(x,u,v)>0\quad \text{and }\quad\frac{\partial f_2}{\partial u}(x,u,v)>0\quad\text{ in  }\Omega\times\RR_+^{\ast}\times\RR_+^{\ast}.\ee
Let $(\underline{u},\underline{v}),(\overline{u},\overline{v})$ $\in\left[\mathscr{C}(\overline{\Omega})\cap\mathrm{W}^{1,p}_{\rm loc}(\Omega)\right]\times\left[\mathscr{C}(\overline{\Omega})\cap \mathrm{W}^{1,q}_{\rm loc}(\Omega)\right]$ be sub and supersolutions pairs to $\boldsymbol{({\rm P})}$. 
Assume in addition that the following conditions hold:
\begin{enumerate}
\item
there exist constants $C_1,C_2>0$ and $b_1,b_2>0$ such that 
\be\label{sssolu}
\overline{u}\leq C_1 d(x)^{b_1}\quad\text{ and }\quad\overline{v}\leq C_2 d(x)^{b_2}\quad\text{in }\Omega,
\ee
\item there exist $\kappa_1,\kappa_2>0$ and $\delta_1,\delta_2\in \RR$ such that
\be\label{partial1}
\left|\frac{\partial f_1}{\partial u}(x,u,v)\right|\leq \kappa_1d(x)^{\delta_1}\;\text{and}\;\left|\frac{\partial f_2}{\partial v}(x,u,v)\right|\leq  \kappa_2d(x)^{\delta_2}\text{ in }\Omega\times\mathcal{C}.
\ee
\end{enumerate}
Then, there exists a positive very weak solution pair $(u,v)\in \left[{\rm L}^{\infty}(\Omega)\cap{\rm W}^{1,p}_{\rm loc}(\Omega)\right]\times \left[{\rm L}^{\infty}(\Omega)\cap{\rm W}^{1,q}_{\rm loc}(\Omega)\right]$  to $\boldsymbol{({\rm P})}$ such that $(u,v)\in \mathcal{C}$.
\end{theorem}
\begin{remark} Since $f_1$ and $f_2$ are continuous with respect to the two last variables in $\RR^{\ast}_+\times\RR^{\ast}_+$, assumptions \eqref{sssolu} and \eqref{partial1} imply that for any $K\subset\subset\Omega$, there exist $C_K, C'_K>0$ such that
\be\label{condf1}
\left| f_1(x,u,v)\right|\leq C_K\quad\text{ and }\quad\left| f_2(x,u,v)\right|\leq C'_K\quad\text{ in }K\times\mathcal{C}.
\ee
\end{remark}
\begin{proof} 
Since $\Omega$ is a smooth domain, we can introduce
$\left( \Omega_n\right)_{n\in\NN^{\ast}}\subset \Omega$ an increasing sequence of smooth subdomains of $\Omega$ such that $\Omega_n\lOngrightarrow\limits_{n\to\infty}\Omega$ in the Hausdorff topology with
\begin{small}$$\forall n\in\NN^{\ast},\quad\frac{1}{n+1}<\mathrm{dist}(\partial\Omega,\partial\Omega_n)<\frac{1}{n}.$$ \end{small}
Then, for all $n\in \NN^{\ast}$ we consider the following iterative scheme: 
$$\boldsymbol{({\rm P}_{n})}\left\lbrace\begin{array}{l}
-\Delta_{p}u_{n}+\kappa_1d(x)^{\delta_1} u_{n} =f_1(x,\tilde{u}_{n-1},\tilde{v}_{n-1})+\kappa_1d(x)^{\delta_1}\tilde{u}_{n-1}\quad\text{ in }\,\Omega_n,\vspace{0.2cm}\\
-\Delta_{q}v_{n}+\kappa_2d(x)^{\delta_2}v_{n} =f_2(x,\tilde{u}_{n-1},\tilde{v}_{n-1})+\kappa_2d(x)^{\delta_2}\tilde{v}_{n-1}\quad\text{ in }\,\Omega_n,\vspace{0.2cm}\\
\quad u_{n}\vert_{\partial\Omega_n}=\underline{u},\quad v_{n}\vert_{\partial\Omega_n}=\underline{v}\quad\text{and}\quad  u_{n}>0, \quad v_{n}>0\quad\text{in } \Omega_n,
 \end{array}
 \right.
$$
with initial data $u_0=\underline{u}$ and $v_0=\underline{v}$ in $\Omega_0$ and for all $n\in\NN$, $$\tilde{u}_n\eqdef {\mathds{1}}_{\Omega_n}.u_n+{\mathds{1}}_{\Omega\setminus\Omega_n}.\underline{u}\quad\text{ and }\quad\tilde{v}_n\eqdef {\mathds{1}}_{\Omega_n}.v_n+{\mathds{1}}_{\Omega\setminus\Omega_n}.\underline{v}\quad\text{ in }\quad\Omega.$$
By induction on $n\in\NN^{\ast}$, $\boldsymbol{({\rm P}_{n})}$ has a solution $(u_n,v_n)\in{\rm W}^{1,p}(\Omega_n)\times{\rm W}^{1,q}(\Omega_n)$ satisfying   for all $n\in\NN^{\ast},$
\be\label{croissance} \underline{u}\leq \tilde{u}_n\leq\tilde{u}_{n+1}\leq\overline{u}\quad \text{ and }\quad \underline{v}\leq \tilde{v}_n\leq\tilde{v}_{n+1}\leq\overline{v}\quad \text{ a.e. in }\Omega.\ee
Indeed, using estimates \eqref{sssolu} and \eqref{condf1}, 
$$f_1(x,\underline{u},\underline{v})+\kappa_1d(x)^{\delta_1}\underline{u}\in {\rm L}^{\infty}(\Omega_1)\hookrightarrow{\rm W}^{-1, p'}(\Omega_1)$$
and since $\underline{u}\in {\rm W}^{1,p}(\Omega_1)$  $\hookrightarrow{\rm W}^{1/p',p}(\partial\Omega_1)$ in the sense of the traces, we get $u_1\in {\rm W}^{1,p}(\Omega_1)$ as a minimum of the functional $E_1$ defined for $w\in{\rm W}^{1,p}(\Omega_1)$ by
\begin{multline}E_1(w)\eqdef \displaystyle\frac{1}{p}\int_{\Omega_1}\vert \nabla(w+\underline{u})\vert^p\,dx
+ \frac{\kappa_1}{2}\int_{\Omega_1}d(x)^{\delta_1}(w+\underline{u})^2\,dx\\
\displaystyle~\;-\int_{\Omega_1} (f_1(x,\underline{u},\underline{v})+\kappa_1 d(x)^{\delta_1}\underline{u})w\,dx.
\end{multline}
Since the operator $u\mapsto -\Delta_pu+\kappa_1d(x)^{\delta_1}u$ is monotone in ${\rm W}^{1,p}(\Omega_1)$, applying the weak comparison principle we get 
$$\underline{u}\leq u_1 \leq \overline{u} \quad\text{a.e. in }\Omega_1.$$ 
Using the same arguments as above, we prove the existence of $v_1\in {\rm W}^{1,q}(\Omega_1)$ satisfying $\underline{v}\leq v_1 \leq \overline{v} $ a.e. in $\Omega_1.$
Now, let us fix $n\in\NN^{\ast}$ and suppose that for all $k\leq n$, $\boldsymbol{({\rm P}_{k})}$ has a solution $(u_k,v_k)\in{\rm W}^{1,p}(\Omega_k)\times{\rm W}^{1,q}(\Omega_k)$ satisfying \eqref{croissance}. The existence of positive solutions of $\boldsymbol{({\rm P}_{n+1})}$, $(u_{n+1},v_{n+1})\in{\rm W}^{1,p}(\Omega_{n+1})\times{\rm W}^{1,q}(\Omega_{n+1})$ satisfying 
$$\underline{u}\leq u_{n+1}\leq\overline{u}\quad\text{ and }\quad\underline{v}\leq v_{n+1}\leq \overline{v}\quad\text{ a.e. in }\Omega_{n+1},$$
 can be established using similar techniques as above.  To prove the monotonicity of the sequences $\left(\tilde{u}_m\right)_{m\in\NN^{\ast}}$ and $\left(\tilde{v}_m\right)_{m\in\NN^{\ast}}$, we remark that $\tilde{u}_n\in {\rm W}^{1,p}(\Omega_{n+1})$ and satisfies 
\be\label{balibali}-\Delta_p\tilde{u}_n+\kappa_1d(x)^{\delta_1}\tilde{u}_n\leq f_1(x,\tilde{u}_{n-1},\tilde{v}_{n-1})+\kappa_1d(x)^{\delta_1}\tilde{u}_{n-1}\quad\text{ in }\Omega_{n+1},\ee
in the weak sense. Then, using \eqref{balibali} together with \eqref{condf1}, we deduce from the previous inequality that,
$$-\Delta_p\tilde{u}_n+\kappa_1d(x)^{\delta_1}\tilde{u}_n\leq f_1(x,\tilde{u}_{n-1},\tilde{v}_{n})+\kappa_1d(x)^{\delta_1}\tilde{u}_{n-1}\quad\text{ in }\Omega_{n+1},$$
in the weak sense. Hence, by estimate \eqref{partial1} and from the weak  comparison principle applied in ${\rm W}^{1,p}(\Omega_{n+1})$, we obtain
$$ \tilde{u}_n\leq u_{n+1}\quad\text{a.e. in }\Omega_{n+1}.$$
Similarly, we get the existence and the behaviour of $v_{n+1}$. 
Then, for almost every $x\in\Omega$, we define 
$$u(x)=\lim\limits_{n\to\infty}\tilde{u}_n(x)\quad\text{ and }\quad v(x)=\lim\limits_{n\to\infty}\tilde{v}_n(x).$$
Moreover, using a classical local regularity result of  {\sc Serrin}\cite{Ser}, $\tilde{u}_n,\tilde{v}_n\in\mathscr{C}^{1,\gamma}_{\rm loc}\left({\Omega}_n\right)$ { for some} $0<\gamma<1$ and $\nabla \tilde{u}_n\lOngrightarrow\limits_{n\to \infty}\nabla u $ and $\nabla \tilde{v}_n\lOngrightarrow\limits_{n\to \infty}\nabla v,$ uniformly  in any compact set  $K$ of $\Omega.$  Thus, $(u,v)\in\left[\underline{u},\overline{u}\right]\times\left[\underline{v},\overline{v}\right]$ and passing to the limit in $\boldsymbol{({\rm P}_{n})}$, $(u,v)$ is a solution of $\boldsymbol{({\rm P})}$ in the sense of distributions.
\end{proof}
\section{Applications}
\subsection{Example 1}\label{ex1}
In this section we focus on the following quasilinear elliptic and singular system, 
$$\boldsymbol{({\rm P})}\left\lbrace 
\begin{array}{ll}
          -\Delta_pu=K_1(x)u^{a_1}v^{b_1}\quad\text{in }\Omega\, ;\quad u|_{\partial\Omega}=0,\quad u>0\quad\text{in }\Omega,\vspace{0.2cm}\\
          -\Delta_qv=K_2(x)v^{a_2}u^{b_2}\quad\text{in }\Omega\, ;\quad v|_{\partial\Omega}=0,\quad v>0\quad\text{in }\Omega.
 \end{array}
\right.
$$
\\~
In this problem, 
\begin{enumerate}
\item The exponents $a_1<p-1$, $a_2<q-1$ and $b_1,b_2\neq 0$ satisfy the subhomogeneous condition
\be\label{H1}{(p-1-a_1)(q-1-a_2)- \vert b_1b_2\vert >0,}\ee
which  is equivalent to the existence of a positive constant $\sigma>0$ such that 
\be\label{condexp1-0}  (p-1-a_1)-\sigma\vert b_1\vert>0\quad\text{ and }\quad \sigma(q-1-a_2)-\vert b_2\vert>0. \ee 
\item $K_1,K_2$ are two positive functions in $\Omega$ { satisfying}
\be\label{K1} K_1(x)=d(x)^{-k_1}L_1(d(x))\;\text{and}\; K_2(x)= d(x)^{-k_2}L_2(d(x))\;\text{in }\Omega,\ee
with $0\leq k_1<p$, $0\leq k_2<q$ and for $i=1,2$, $L_i$ a lower perturbation in $\mathscr{C}^{2}((0,D])$ ($ D$ the diameter of the domain $\Omega$), of the form:
\be\label{L}\forall t\in (0,D],\; L_i(t)=\exp\left(\int_t^{2D}\frac{z_i(s)}{s}ds\right),\ee
with $ z_i\in\mathscr{C}([0,D])\cap \mathscr{C}^1((0,D])$ and $z_i(0)=0$.
\begin{remark}\label{PropL}
\begin{enumerate}
\item Let us notice that \eqref{L} implies that 
\be\label{ln}\forall \varepsilon>0,\qquad \lim\limits_{t\to 0^+}t^{-\varepsilon}L_i(t)=+\infty\quad\text{ and }\quad \lim\limits_{t\to 0^+}t^{\varepsilon}L_i(t)=0.\ee
\item Definition \eqref{L} also implies that 
\be\label{tL'/L}\lim\limits_{t\to 0^+}\frac{tL_i'(t)}{L_i(t)}=0\quad\text{ and }\quad\lim\limits_{t\to 0^+}\frac{tL_i''(t)}{L_i'(t)}=-1\nonumber.\ee
\item If $L_1,L_2$ are two functions satisfying \eqref{L}, then for any $\alpha,\beta\in\RR$, the function ${L_1}^{\alpha}.{L_2}^{\beta}$ also satisfies \eqref{L}.
\item Such functions $L_1,$ $L_2$ defined as above belong to the Karamata Class \cite{Kar}.
\end{enumerate}
\end{remark}
\begin{example} Let $m\in\NN^{\ast}$ and $A>>D$ large enough. Let us define 
\begin{small}$$\forall t\in(0,D],\qquad \displaystyle L_i(t)=\prod_{n=1}^m\left(\log_n\left(\frac{A}{t}\right)\right)^{\mu_n},$$                                   \end{small} where, $\log_n\eqdef\log \circ\cdots \circ\log$ ($n$ times) and $\mu_n>0$. Then $L_i$ satisfies \eqref{L}.
\end{example}
\end{enumerate}
In our study, $b_1\neq 0$ and $b_2\neq 0$. In the case where $b_1>0$ and $b_2>0$, the expression of the right-hand sides of the two coupled equations in system $\textbf{(P)}$ define a \textbf{\textit{cooperative}} interaction between the two components (species) $u$ and $v$: 
\be\label{dp1}\frac{\partial}{\partial v}\left(K_1(x)u^{a_1}v^{b_1}\right)=b_1K_1(x)u^{a_1}v^{b_1-1}>0,\ee
\be\label{dp2}\frac{\partial}{\partial u}\left(K_2(x)v^{a_2}u^{b_2}\right)=b_2K_2(x)v^{a_2}u^{b_2-1}>0.\ee
In the case where $b_1<0$ and $b_2<0$, both partial derivative in \eqref{dp1} and \eqref{dp2} are negative and the expression of the right-hand sides of the two coupled equations of $\textbf{(P)}$ defines a \textbf{\textit{competitive}} interaction between $u$ and $v$.  
\\

First, we discuss the existence of positive weak solutions pairs to problem $\boldsymbol{({\rm P})}$. For that, regarding Theorem \ref{TH1},  we take $$f_1(x,u,v)=K_1(x)u^{a_1}v^{b_1},\quad f_2(x,u,v)=K_2(x)v^{a_2}u^{b_2}$$ and construct suitable sub and supersolutions pairs of $\textbf{(P)}$ in ${\rm W}^{1,p}_0(\Omega)\times{\rm W}^{1,q}_0(\Omega)$.
 
Then, in the cases where $\boldsymbol{({\rm P})}$ is either competitive or cooperative, we investigate the uniqueness of such positive weak solutions pairs. For that, it is essential that the mappings $T_1\circ T_2$ and $T_2\circ T_1$ (where $T_1$ and $T_2$ are defined in \eqref{T}) is \textit{\textbf{subhomogeneous}}, which is equivalent to condition \eqref{H1}.

\subsubsection{Preliminary results}

Let $1<r<\infty$, $\delta<r-1$ and $K:x\longmapsto d(x)^{-k}L(d(x))$, with $0\leq k<r$ and $L$ a perturbation function satisfying \eqref{L}. In view of constructing suitable pairs of sub and supersolutions to $\boldsymbol{({\rm P})}$, we first introduce the following problem:
\be\label{eq}-\Delta_r w=K(x)w^{\delta}\quad\text{ in }\Omega;\quad w\vert_{\partial\Omega}=0,\quad w>0\quad\text{ in }\Omega.\ee
\begin{theorem}\label{gms}
Under the above hypothesis, we have:
\begin{enumerate}
\item If $k-1<\delta<r-1$, problem \eqref{eq} has a unique positive weak solution $\psi\in {\rm W}^{1,r}_0(\Omega)$ that satisfies the following estimate:
\be \psi(x)\,\sim\,d(x)\qquad\text{in }\Omega.\ee
In addition, we have $\psi\in\mathscr{C}^{1,\alpha}\left(\overline{\Omega}\right)$, for some $0<\alpha<1$.
\\
\item If $\delta=k-1$, problem \eqref{eq} has a unique positive weak solution $\psi\in \mathrm{W}^{1,r}_0(\Omega)$ that satisfies the following estimate:
 \be\label{critique}\psi(x)\,\sim\,d(x)\left(\int^{2D}_{d(x)}L(t)t^{-1}\,dt\right)^{\frac{1}{r-k}}\qquad\text{in }\Omega.\ee
 In addition, we have $\psi\in\mathscr{C}^{0,\alpha}\left(\overline{\Omega}\right)$, for some $0<\alpha<1$.
 \\
\item If $k-2+\frac{k-1}{r-1}<\delta<k-1$, problem \eqref{eq} has a unique positive weak solution $\psi\in \mathrm{W}^{1,r}_0(\Omega)$ that satisfies the following estimate:
\be \psi(x)\,\sim\,d(x)^{\frac{r-k}{r-1-\delta}}L(d(x))^{\frac{1}{r-1-\delta}}\qquad\text{in }\Omega.\ee
In addition, we have $\psi\in\mathscr{C}^{0,\alpha}\left(\overline{\Omega}\right)$, for some $0<\alpha<1$.
\\
\item If $\delta\leq k-2+\frac{k-1}{r-1}$, problem \eqref{eq} has at least one positive weak solution $\psi\in \mathrm{W}^{1,r}_{{\rm loc}}(\Omega)\cap\mathscr{C}_0\left(\overline{\Omega}\right)$ that satisfies the following estimate:
\be \psi(x)\,\sim\,d(x)^{\frac{r-k}{r-1-\delta}}L(d(x))^{\frac{1}{r-1-\delta}}\qquad\text{in }\Omega.\ee
 \end{enumerate}
\end{theorem}
\begin{proof} 
See Lemma 3.3  in  {\sc Giacomoni, M\^aagli, Sauvy}\cite{GMS}.
\end{proof}
\begin{remark}
In (iv) above, it can be proved that  $\forall\,\gamma>$\footnotesize $\frac{(r-1)(r-1-\delta)}{r(r-k)}$, ${\psi}^{\gamma}\in {\rm W}^{1,r}_0(\Omega)$.
\end{remark}
We give now a weak comparison principle used  to establish the uniqueness of a positive weak solutions pair of $\textbf{(P)}$.
\begin{theorem}\label{wmp}
Let $K:\Omega\rightarrow\RR_+$ be a ${\rm L}^{1}_{{\rm loc}}(\Omega)$ function and $\delta<r-1$. Assume  $u,v\in\mathrm{W}^{1,r}_0(\Omega)\cap{\rm L}^{\infty}(\Omega)$ are two locally uniformly positive functions satisfying the sub and supersolution inequalities:
\be\label{leq}-\Delta_ru\leq K(x)u^{\delta}\quad and \quad -\Delta_rv\geq K(x)v^{\delta}\qquad in\; \Omega,\ee
in the sense of distributions (i.e. Radon measures) in $\mathrm{W}^{-1,r'}(\Omega)$. Then
\begin{enumerate}
\item If $\delta<0$, inequality $u\leq v$ holds a.e. in $\Omega$.
\item If $\delta>0$ and if we suppose in addition that there exist $C_1,C_2>0$ and a locally uniformly positive function $w_0\in {\rm L}^{\infty}(\Omega)$ such that $C_1 w_0\leq u,v\leq C_2 w_0$ a.e. in $\Omega$ and 
\be\label{condinteg} \int_{\Omega}K(x){w_0}^{\delta+1}\,dx<+\infty,\ee
inequality $u\leq v$ holds a.e. in $\Omega$.
\end{enumerate}
\end{theorem}
To prove this theorem, we use the well-known  inequality { in Lemma \ref{Peral}}  and the  D\'iaz-Saa inequality (see {\sc D\'iaz-Saa}\cite{Diaz}).
\begin{lemma}\label{Peral}There exists a constant $C_r>0$ such that, for all $x,y\in\RR^N$,
\begin{small} $$\vert x\vert ^r-\vert y\vert^r-r\vert x\vert^{r-2}x.(y-x)\geq\left\lbrace\begin{array}{l}\displaystyle \quad C_r\vert x-y\vert^{r}\quad if\quad r\geq 2,\\\displaystyle
C_r\frac{\vert x-y\vert^2}{(\vert x\vert+\vert y\vert)^{2-r}}\quad if\quad 1<r<2.
\end{array}\right.$$\end{small}
\end{lemma}
\begin{proof} See  Lemma 4.2 in \textsc{Lindqvist} \cite{Lin}.\end{proof}
\begin{proof}\textbf{{\sc (of Theorem \ref{wmp})}}
\begin{enumerate}
\item If $\delta<0$, we wish to prove that the function $w=(u-v)^+$ satisfies $w=0$ a.e. in $\Omega$. First notice that $0\leq w\in {\rm W}^{1,r}_0(\Omega)$. Applying the duality between ${\rm W}^{1,r}_0(\Omega)$ and ${\rm W}^{-1,r'}(\Omega)$, respectively, to $w$ and the the difference 
$$-\Delta_ru+\Delta_r v\leq K(x)\left(u^{\delta}-v^{\delta}\right)$$
which is $\leq0$ on the set $\Omega_+\eqdef\lbrace x\in\Omega,\quad w(x)>0\rbrace$, we obtain
\begin{multline*}
\displaystyle\int_{\Omega_+}\left(\vert\nabla u\vert^{r-2}\nabla u-\vert \nabla v\vert^{r-2}\nabla v\right).(\nabla u-\nabla v)\,dx\\
\displaystyle=\int_{\Omega}\left(\vert\nabla u\vert^{r-2}\nabla u-\vert \nabla v\vert^{r-2}\nabla v\right).\nabla w\,dx\leq 0.
\end{multline*}
This forces $\nabla w=0$ a.e. in $\Omega_+$ and, consequently, also in $\Omega$. Since $w\in{\rm W}^{1,r}_0(\Omega)$, we conclude that $w=0$ a.e. in $\Omega$ as claimed, that is, $u\leq v$ a.e. in $\Omega$.
\item If $0<\delta <r-1$, 
following some ideas in {\sc Lindquist}\cite{Lin} (see also {\sc Dr\'abek-Hern\'andez}\cite{dra-her}), we use the D\'iaz-Saa inequality.\\
More precisely, for $\varepsilon >0$, we set $u_{\varepsilon}\eqdef u+\varepsilon$ and $v_{\varepsilon}\eqdef v +\varepsilon$ in $\Omega$ and we define
\be \phi\eqdef\frac{{u_{\varepsilon}}^r-{v_{\varepsilon}}^r}{{u_{\varepsilon}}^{r-1}}\quad\text{ and }\quad\psi\eqdef\frac{{v_{\varepsilon}}^r-{u_{\varepsilon}}^r}{{v_{\varepsilon}}^{r-1}}\quad\text{ in }\Omega.\nonumber\ee
Then, $\frac{u_{\varepsilon}}{v_{\varepsilon}},\frac{v_{\varepsilon}}{u_{\varepsilon}}\in{\rm L}^{\infty}(\Omega)$ and $\phi,\psi\in{\rm W}^{1,r}_0(\Omega)$ with
\be\label{gradphi} \nabla\phi=\left[ 1+(r-1)\left(\frac{v_{\varepsilon}}{u_{\varepsilon}}\right)^r\right]\nabla u-r\left(\frac{v_{\varepsilon}}{u_{\varepsilon}}\right)^{r-1}\nabla v\quad\text{ in }\Omega,\ee
\be\label{gradpsi} \nabla\psi=\left[ 1+(r-1)\left(\frac{u_{\varepsilon}}{v_{\varepsilon}}\right)^r\right]\nabla v-r\left(\frac{u_{\varepsilon}}{v_{\varepsilon}}\right)^{r-1}\nabla u\quad\text{ in }\Omega.\ee
Setting $\Omega_+\eqdef\lbrace x\in \Omega,\quad u(x)>v(x)\rbrace$, we have that $\phi>0$ and $\psi<0$ in $\Omega_+$ and 
\be\label{phiomega+} \int_{\Omega_+}\vert\nabla u\vert ^{r-2}\nabla u.\nabla \phi\,dx\leq \int_{\Omega_+}K(x)u^{\delta}\phi\, dx<+\infty,\nonumber\ee
\be\label{psiomega+} \int_{\Omega_+}\vert\nabla v\vert ^{r-2}\nabla v.\nabla \psi\,dx\leq \int_{\Omega_+}K(x)v^{\delta}\psi\, dx<+\infty.\nonumber\ee
Using equalities \eqref{gradphi} and \eqref{gradpsi} and the fact that \be \vert \nabla\ln u_{\varepsilon}\vert=\frac{\vert \nabla u\vert}{u_{\varepsilon}}\quad\text{ and }\quad\vert \nabla\ln v_{\varepsilon}\vert=\frac{\vert \nabla v\vert}{v_{\varepsilon}}\quad\text{ in }\Omega,\ee we get
\be \begin{array}{l}
\displaystyle\int_{\Omega_+}\vert\nabla u\vert ^{r-2}\nabla u.\nabla \phi\,dx+\int_{\Omega_+}\vert\nabla v\vert ^{r-2}\nabla v.\nabla \psi\,dx\qquad\qquad~\\
\displaystyle~\qquad\qquad=\int_{\Omega_+}({u_{\varepsilon}}^r-{v_{\varepsilon}}^r)(\vert \nabla\ln u_{\varepsilon}\vert^r-\vert \nabla\ln v_{\varepsilon}\vert^r)\,dx\\
\displaystyle~\qquad\qquad-\int_{\Omega_+} r{v_{\varepsilon}}^r\vert \nabla \ln u_{\varepsilon}\vert^{r-2}(\nabla\ln u_{\varepsilon}).\left(\nabla\ln v_{\varepsilon}-\nabla\ln u_{\varepsilon}\right)\,dx\\
\displaystyle~\qquad\qquad-\int_{\Omega_+} r{u_{\varepsilon}}^r\vert \nabla \ln v_{\varepsilon}\vert^{r-2}(\nabla\ln v_{\varepsilon}).\left(\nabla\ln u_{\varepsilon}-\nabla\ln v_{\varepsilon}\right)\,dx.\\\end{array}
\nonumber\ee
\begin{enumerate}
\item
If $r\geq2$, from Lemma \ref{Peral}, it follows that 
\be \begin{array}{l}
\displaystyle\int_{\Omega_+}\vert\nabla u\vert ^{r-2}\nabla u.\nabla \phi\,dx+\int_{\Omega_+}\vert\nabla v\vert ^{r-2}\nabla v.\nabla \psi\,dx~\qquad\qquad~\\
\displaystyle~\quad\;\geq\int_{\Omega_+}({u_{\varepsilon}}^r-{v_{\varepsilon}}^r)(\vert \nabla\ln u_{\varepsilon}\vert^r-\vert \nabla\ln v_{\varepsilon}\vert^r)\,dx\\
\displaystyle~\quad\;+\int_{\Omega_+} {v_{\varepsilon}}^r\left(\vert \nabla \ln u_{\varepsilon}\vert^{r}-\vert \nabla \ln v_{\varepsilon}\vert^{r}+C_r\vert \nabla \ln u_{\varepsilon}-\nabla \ln v_{\varepsilon}\vert^{r}\right)\,dx\\
\displaystyle~\quad\;+\int_{\Omega_+} {u_{\varepsilon}}^r\left(\vert \nabla \ln v_{\varepsilon}\vert^{r}-\vert \nabla \ln u_{\varepsilon}\vert^{r}+C_r\vert \nabla \ln u_{\varepsilon}-\nabla \ln v_{\varepsilon}\vert^{r}\right)\,dx\\
\displaystyle~\quad\;=C_r\int_{\Omega_+}\vert u_{\varepsilon}\nabla v_{\varepsilon}-v_{\varepsilon}\nabla u_{\varepsilon}\vert^r\left(\frac{1}{{u_{\varepsilon}}^r}+\frac{1}{v_{\varepsilon}^r}\right)\,dx.
\end{array}
\nonumber\ee
\item If $1<r<2$, Lemma \ref{Peral} { entails}
\begin{multline*}
\displaystyle\int_{\Omega_+}\vert\nabla u\vert ^{r-2}\nabla u.\nabla \phi\,dx+\int_{\Omega_+}\vert\nabla v\vert ^{r-2}\nabla v.\nabla \psi\,dx
\\
\geq
C_r\int_{\Omega_+}\frac{\vert u_{\varepsilon}\nabla v_{\varepsilon}-v_{\varepsilon}\nabla u_{\varepsilon}\vert^2}{(u_{\varepsilon}\vert\nabla v_{\varepsilon}\vert+v_{\varepsilon}\vert\nabla u_{\varepsilon}\vert)^{2-r}}\left(\frac{1}{{u_{\varepsilon}}^r}+\frac{1}{{v_{\varepsilon}}^r}\right)\,dx.
\end{multline*}
\end{enumerate}
In the right-hand side, we get
\begin{multline*}
\displaystyle\int_{\Omega_+}K(x)\left(u^{\delta}\phi+v^{\delta}\psi\right)\,dx=\\
\int_{\Omega_+}K(x)\left[\frac{u^{\delta}}{u^{r-1}}\left(\frac{u}{u_{\varepsilon}}\right)^{r-1}-\frac{v^{\delta}}{v^{r-1}}\left(\frac{v}{v_{\varepsilon}}\right)^{r-1}\right]
\left({u_{\varepsilon}}^r-{v_{\varepsilon}}^r\right)\,dx.
\end{multline*}
Then, since $\frac{u}{u_{\varepsilon}}\to1$, $\frac{v}{v_{\varepsilon}}\to 1$ as $\varepsilon\to 0^+$ a.e. in $\Omega$, we get from \eqref{condinteg} and Lebesgue's Theorem that
\be \lim\limits_{\varepsilon\to 0^+}\int_{\Omega_+}K(x)\left(u^{\delta}\phi+v^{\delta}\psi\right)\,dx\leq 0.\nonumber\ee
By Fatou's Lemma and using the above estimates, we obtain in the both cases that $\vert u\nabla v-v\nabla u\vert=0$ a.e. in $\Omega_+$, from which we get that on each connected component set $\omega$ of $\Omega_+$, there exists $k>0$ such that $u=kv$ a.e. in $\omega$.
From sub an supersolution inequalities \eqref{leq} we have,
\begin{multline}
\displaystyle k^{r}\int_{\omega}K(x) v^{\delta+1}\,dx\leq\displaystyle k^r\int_{\omega}\vert \nabla v\vert^r\,dx=\displaystyle\int_{\omega} \vert \nabla u\vert^r\,dx \\
\leq \displaystyle\int_{\omega}K(x) u^{\delta+1}\,dx= k^{\delta+1}\int_{\omega} K(x) v^{\delta+1}\,dx.
\end{multline}
Hence, $k\leq 1$ which implies that $u\leq v$ a.e. in $\Omega_+$ and from the definition of $ \Omega_+$, $u\leq v$ a.e. in $\Omega.$
\end{enumerate}
\end{proof}

\subsubsection{Main results}
\begin{theorem}\label{theo}
Assume that the exponents $a_1<p-1$, $a_2<q-1$ and $b_1,b_2\neq0$ in problem $\boldsymbol{({\rm P})}$ satisfy the hypothesis \eqref{H1}.
\begin{enumerate}
\item Set
\begin{footnotesize}
\be\label{alpha}  \alpha_1=\frac{q-1-a_2}{(p-1-a_1)(q-1-a_2)-b_1b_2},\quad \alpha_2=\frac{p-1-a_1}{(p-1-a_1)(q-1-a_2)-b_1b_2},\nonumber\ee
\end{footnotesize}
\begin{footnotesize}
\be\label{beta} \beta_1=\frac{b_1}{(p-1-a_1)(q-1-a_2)-b_1b_2},\quad\beta_2=\frac{b_2}{(p-1-a_1)(q-1-a_2)-b_1b_2},\nonumber\ee
\end{footnotesize}
\\
\begin{footnotesize}
\be\label{gamma1} \gamma_1=\frac{(p-k_1)(q-1-a_2)+(q-k_2)b_1}{(p-1-a_1)(q-1-a_2)-b_1b_2},\quad\gamma_2=\frac{(q-k_2)(p-1-a_1)+(p-k_1)b_2}{(p-1-a_1)(q-1-a_2)-b_1b_2}\nonumber\ee
\end{footnotesize}
\\
and assume that 
\be\label{gamma} {1-\frac{1}{p}<\gamma_1 <1}\quad\text{ and }\quad{1-\frac{1}{q}<\gamma_2 <1}.  \ee
Then, problem $\boldsymbol{({\rm P})}$ possesses  positive solutions $(u,v)\in {\rm W}^{1,p}_0(\Omega)\times{\rm W}^{1,q}_0(\Omega)$ that satisfy the following estimates:
\be\label{est1u} u(x) \,\sim\,d(x)^{\gamma_1}L_1(d(x))^{\alpha_1}L_2(d(x))^{\beta_1}\quad\text{ in }\Omega,\ee
\be\label{est1v} v(x)\,\sim\,d(x)^{\gamma_2}L_2(d(x))^{\alpha_2}L_1(d(x))^{\beta_2}\quad\text{ in }\Omega.\ee 

In addition, we have $(u,v)\in\mathscr{C}^{0,\alpha}\left(\overline{\Omega}\right)\times\mathscr{C}^{0,\alpha}\left(\overline{\Omega}\right)$, for some $0<\alpha<1$.
\\
\item Now assume that
\be\label{gamma2} k_1-1<a_1+b_1<p-1\quad\text{ and }\quad k_2-1< a_2+b_2<q-1.\ee
Then, problem $\mathrm{\boldsymbol{(P)}}$ possesses  positive  solutions $(u,v)\in {\rm W}^{1,p}_0(\Omega)\times{\rm W}^{1,q}_0(\Omega)$ that satisfy the following estimates:
\be\label{est2}  u(x)\,\sim\,d(x) \quad\text{ and } \quad  v(x)\,\sim\,d(x)\quad\text{ in }\Omega.\ee
In addition, we have $(u,v)\in\mathscr{C}^{1,\alpha}\left(\overline{\Omega}\right)\times\mathscr{C}^{1,\alpha}\left(\overline{\Omega}\right)$, for some $0<\alpha<1$.
\\
\item Set
\be\gamma=\frac{p-k_1+b_1}{p-1-a_1}\nonumber\ee
and assume that
\be \label{gamma3}1-\frac{1}{p}<\gamma<1\quad\text{ and }\quad k_2-1<a_2+b_2\gamma<q-1.\ee
Then, problem $\mathrm{\boldsymbol{(P)}}$ possesses positive solutions $(u,v)\in {\rm W}^{1,p}_0(\Omega)\times{\rm W}^{1,q}_0(\Omega)$ that satisfy the following estimates:
\be \label{est3} u(x)\,\sim\,d(x)^{\gamma}L_1(d(x))^{\frac{1}{p-1-a_1}} \quad\text{ and } \quad  v(x)\,\sim\,d(x)\quad\text{ in }\Omega.\ee
In addition, we have $(u,v)\in\mathscr{C}^{0,\alpha}\left(\overline{\Omega}\right)\times\mathscr{C}^{1,\alpha}\left(\overline{\Omega}\right)$, for some $0<\alpha<1$.
\\
\item Symmetrically to part (iii) above, set
\be \gamma=\frac{q-k_2+b_2}{q-1-a_2}\nonumber\ee
and assume that
\be\label{gamma4} k_1-1<a_1+b_1\gamma<p-1\quad\text{ and }\quad 1-\frac{1}{q}<\gamma<1.\ee
Then, problem $\mathrm{\boldsymbol{(P)}}$ possesses  positive  solutions $(u,v)\in {\rm W}^{1,p}_0(\Omega)\times{\rm W}^{1,q}_0(\Omega)$ that satisfy the following estimates:
\be \label{est4} u(x)\,\sim\,d(x) \quad\text{ and } \quad  v(x)\,\sim\,d(x)^{\gamma}L_2(d(x))^{\frac{1}{q-1-a_2}}\quad\text{ in }\Omega.\ee
In addition, we have $(u,v)\in\mathscr{C}^{1,\alpha}\left(\overline{\Omega}\right)\times\mathscr{C}^{0,\alpha}\left(\overline{\Omega}\right)$, for some $0<\alpha<1$.
\end{enumerate}
\end{theorem}
\begin{theorem}\label{uniq}
Let $a_1<p-1$, $a_2<q-1$ and $b_1,b_2\neq 0$  satisfying the subhomogeneity hypothesis \eqref{H1}. Assume  that $\mathrm{\boldsymbol{(P)}}$ is either a competitive or a cooperative system, \textit{i.e.} $\boldsymbol{b_1b_2>0}$. Then, each solution provided by Theorem \ref{theo} is unique.
\end{theorem}
{ The cooperative case is further analysed in the following result:}
\begin{theorem}\label{cor}
Let us suppose that the exponents $a_1<p-1$, $a_2<q-1$ and $b_1,b_2\neq 0$  satisfy the subhomogeneity hypothesis \eqref{H1}.
Moreover, assume  that $\mathrm{\boldsymbol{(P)}}$ is a cooperative system, \textit{i.e.}, $\boldsymbol{b_1>0\text{ and }b_2>0}$.
\begin{enumerate}
\item Set
  {\small \be \gamma_1=\frac{(p-k_1)(q-1-a_2)+(q-k_2)b_1}{(p-1-a_1)(q-1-a_2)-b_1b_2},\ee }  {\small \be\gamma_2=\frac{(q-k_2)(p-1-a_1)+(p-k_1)b_2}{(p-1-a_1)(q-1-a_2)-b_1b_2}\ee}
and assume that one of the three following conditions are satisfied:
\be 0<\gamma_1\leq 1-\frac{1}{p}\quad\text{ and }\quad 0<\gamma_2\leq 1-\frac{1}{q},\ee
\be 1-\frac{1}{p}<\gamma_1<1\quad\text{ and }\quad 0<\gamma_2\leq 1-\frac{1}{q},\ee
\be 0<\gamma_1\leq 1-\frac{1}{p}\quad\text{ and }\quad 1-\frac{1}{q}<\gamma_2<1.\ee
Then, problem $\boldsymbol{({\rm P})}$ admits positive solutions $(u,v)\in {\rm W}^{1,p}_{\rm loc}(\Omega)\times{\rm W}^{1,q}_{\rm loc}(\Omega)$ in the sense of distributions satisfying the estimates \eqref{est1u} and \eqref{est1v}. 
\item Set
\be \gamma=\frac{p-k_1+b_1}{p-1-a_1}\ee
and assume that
\be0<\gamma\leq 1-\frac{1}{p}\quad\text{ and }\quad k_2-1<a_2+b_2\gamma<q-1.\ee
Then, problem $\boldsymbol{({\rm P})}$ nevertheless admits positive solutions $(u,v)\in {\rm W}^{1,p}_{\rm loc}(\Omega)\times{\rm W}^{1,q}_0(\Omega)$ in the sense of distributions satisfying the estimates \eqref{est1u} and \eqref{est1v}. 
\\
\item Symmetrically to part (ii) above, set
\be \gamma=\frac{q-k_2+b_2}{q-1-a_2}\ee
and assume that 
\be k_1-1<a_1+b_1\gamma<p-1\quad\text{ and }\quad 0<\gamma\leq 1-\frac{1}{q}.\ee
Then, problem $\mathrm{\boldsymbol{(P)}}$ possesses  positive solutions $(u,v)\in {\rm W}^{1,p}_0(\Omega)\times{\rm W}^{1,q}_{\rm loc}(\Omega)$ in the sense of distributions that satisfies the estimates given in \eqref{est4}.
\end{enumerate}
\end{theorem}
{The next result deals with some limiting cases:}
\begin{theorem}\label{theo2} 
Assume that the exponents $a_1<p-1, a_2<q-1$ and $b_1,b_2\neq0$ satisfy the subhomogeneity hypothesis \eqref{H1}.
\begin{enumerate}
\item Assume that \be a_1+b_1=k_1-1\quad\text{ and }\quad k_2-1\leq a_2+b_2< q-1.\ee 
Then, for all $\varepsilon>0$ small enough, there exist  $C_1,C_2>0$ and $C_1',C_2'>0$ such that problem $\mathrm{\boldsymbol{(P)}}$ possesses  positive  solutions $(u,v)\in{\rm W}^{1,p}_0(\Omega)\times{\rm W}^{1,q}_0(\Omega)$ that satisfy the following estimates:
\be C_1d(x)\leq u \leq C_2 d(x)^{1-\varepsilon}\text{ and } C_1'd(x) \leq v\leq C_2' d(x)^{1-\varepsilon\sigma}\text{ in }\Omega,\ee
where $\sigma>0$ is given in \eqref{condexp1-0}.In addition, we have $(u,v)\in \mathscr{C}^{0,\alpha}\left(\overline{\Omega}\right)\times\mathscr{C}^{0,\alpha}\left(\overline{\Omega}\right)$, for some $0<\alpha<1$.
\item Symmetrically, assume that 
\be a_2+b_2=k_2-1\quad\text{ and }\quad k_1-1\leq a_1+b_1<q-1.\ee 
Then, for all $\varepsilon>0$ small enough, there exist  $C_1,C_2>0$ and $C_1',C_2'>0$ such that problem $\mathrm{\boldsymbol{(P)}}$ possesses  positive  solutions $(u,v)\in{\rm W}^{1,p}_0(\Omega)\times{\rm W}^{1,q}_0(\Omega)$ that satisfy the following estimates:
\be C_1d(x)\leq u \leq C_2 d(x)^{1-\varepsilon}\text{ and } C_1'd(x) \leq v\leq C_2' d(x)^{1-\varepsilon\sigma}\text{ in }\Omega.\ee
In addition, we have $(u,v)\in \mathscr{C}^{0,\alpha}\left(\overline{\Omega}\right)\times\mathscr{C}^{0,\alpha}\left(\overline{\Omega}\right)$, for some $0<\alpha<1$.
\item Let us abbreviate 
\be \gamma=\frac{p-k_1+b_1}{p-1-a_1}\nonumber\ee
and assume that 
\be\label{garage}
1-\frac{1}{p}<\gamma<1\quad\;and\;\quad a_2+b_2\gamma=k_2-1.
\ee
Then, for all $\varepsilon>0$ small enough, there exist  $C_1,C_2>0$ and $C_1',C_2'>0$ such that problem $\mathrm{\boldsymbol{(P)}}$ possesses  positive  solutions $(u,v)\in{\rm W}^{1,p}_0(\Omega)\times{\rm W}^{1,q}_0(\Omega)$ that satisfy the following estimates in $\Omega$:
\be C_1d(x)^{\gamma+\varepsilon}\leq u \leq C_2 d(x)^{\gamma-\varepsilon}\text{ and } C_1'd(x) \leq v\leq C_2' d(x)^{1-\varepsilon\sigma}.\ee
In addition, we have $(u,v)\in \mathscr{C}^{0,\alpha}\left(\overline{\Omega}\right)\times\mathscr{C}^{0,\alpha}\left(\overline{\Omega}\right)$, for some $0<\alpha<1$.
\item Symmetrically, let us abbreviate 
\be \gamma=\frac{q-k_2+b_2}{q-1-a_2}\nonumber\ee
and assume that 
\be\label{garage2}
a_1+b_1\gamma=k_2-1\quad\;and\;\quad 1-\frac{1}{q}<\gamma<1 .
\ee
Then, for all $\varepsilon>0$ small enough, there exist  $C_1,C_2>0$ and $C_1',C_2'>0$ such that problem $\mathrm{\boldsymbol{(P)}}$ possesses  positive  solutions $(u,v)\in{\rm W}^{1,p}_0(\Omega)\times{\rm W}^{1,q}_0(\Omega)$ that satisfy the following estimates in $\Omega$:
\be C_1d(x)\leq u \leq C_2 d(x)^{1-\varepsilon}\text{ and } C_1'd(x)^{\gamma+\varepsilon\sigma} \leq v\leq C_2' d(x)^{\gamma-\varepsilon\sigma}.\ee
In addition, we have $(u,v)\in \mathscr{C}^{0,\alpha}\left(\overline{\Omega}\right)\times\mathscr{C}^{0,\alpha}\left(\overline{\Omega}\right)$, for some $0<\alpha<1$.
\end{enumerate}
\end{theorem}
\subsubsection{Proof of Theorem \ref{theo}}
{ Thanks to Theorem \ref{gms}, we apply Theorem \ref{TH1} with a suitable choice of sub and supersolutions pairs $(\underline{u},\underline{v}),(\overline{u},\overline{v})\in{\rm W}^{1,p}_0(\Omega)\times{\rm W}^{1,q}_0(\Omega)$} { in the following form:}
\be \underline{u}\equiv m \psi_{1}\quad\text{ and }\quad \overline{u}\equiv m^{-1}\psi_{1}\quad\text{ in }\Omega,\nonumber\ee
\be \underline{v}\equiv m^{\sigma} \psi_{2}\quad\text{ and }\quad \overline{v}\equiv m^{-\sigma}\psi_{2}\quad\text{ in }\Omega,\nonumber\ee
where $\sigma>0$ is given in \eqref{condexp1-0}, $0<m<1$ is an appropriate constant small enough and  $\psi_{1}\in{\rm W}^{1,p}_0(\Omega)$, $\psi_{2}\in{\rm W}^{1,q}_0(\Omega)$ are given by Theorem \ref{gms} as the respective unique solutions of problems
\be\label{sssoluex1} -\Delta_p w=d(x)^{-k_1}\mathcal{L}_1(d(x))w^{\delta_1}\quad\text{ in }\Omega;\quad w\vert_{\partial\Omega}=0,\quad w>0\quad \text{ in }\Omega,\ee
\be\label{sssolvex1} -\Delta_q w=d(x)^{-k_2}\mathcal{L}_2(d(x))w^{\delta_2}\quad\text{ in }\Omega;\quad w\vert_{\partial\Omega}=0,\quad w>0\quad \text{ in }\Omega,\ee
satisfying some cone conditions we specify below. In the following {alternatives}, we choose suitable perturbations $\mathcal{L}_1,\mathcal{L}_2$ as in \eqref{L} and  suitable values of exponents $k_1-2+\frac{k_1-1}{p-1}<\delta_1<p-1$ and $k_2-2+\frac{k_2-1}{q-1}<\delta_2<q-1$ in order to satisfy 
\be\label{estsssolex1}  -\Delta_p\psi_1\sim K_1(x){\psi_1}^{a_1}{\psi_2}^{b_1}\quad\text{ and }\quad-\Delta_q{\psi_2}\sim K_2(x){\psi_2}^{a_2}{\psi_1}^{b_2}\quad\text{ in }\Omega,\ee
which provide us the inequalities \eqref{soussol} to \eqref{sursol} {in order} to apply Theorem \ref{TH1}.

\textbf{\textit{Alternative 1:}} We look for positive solutions  $(u,v)\in{\rm W}^{1,p}_0(\Omega)\times{\rm W}^{1,q}_0(\Omega)$ to $\boldsymbol{({\rm P})}$ by making the "Ansatz" that 
\be u(x)\sim d(x)^{\gamma_1}L_1(d(x))^{\alpha_1}L_2(d(x))^{\beta_1}\quad\text{ in }\Omega,\nonumber\ee
\be v(x)\sim d(x)^{\gamma_2}L_2(d(x))^{\alpha_2}L_1(d(x))^{\beta_2}\quad\text{ in }\Omega,\nonumber\ee
for some $\gamma_1\in(1-\frac{1}{p},1)$, $\gamma_2\in(1-\frac{1}{q},1)$ and $\alpha_1,\alpha_2,\beta_1,\beta_2\in\RR$.
For that, we take in \eqref{sssoluex1} and \eqref{sssolvex1}
\be\label{conddelta} k_1-2+\frac{k_1-1}{p-1}<\delta_1< k_1-1\quad\text{ and }\quad\text k_2-2+\frac{k_2-1}{q-1}<\delta_2< k_2-1,\ee 
\be \mathcal{L}_1={L_1}^{\lambda_1}.{L_2}^{\mu_1}\quad \text{ and }\quad\mathcal{L}_2={L_2}^{\lambda_2}.{L_1}^{\mu_2}\quad\text{ in }\Omega,\nonumber\ee
where $\lambda_1,\,\lambda_2,\,\mu_1,\,\mu_2\in \RR$ are suitable exponents we {fix} later.
By Theorem \ref{gms}, $\psi_1\in{\rm W}^{1,p}_0(\Omega)$, $\psi_2\in{\rm W}^{1,q}_0(\Omega)$ and satisfy
\be\label{psi1ex1}\psi_{1}(x)\,\sim \,d(x)^{\frac{p-k_1}{p-1-\delta_1}}L_1(d(x))^{\frac{\lambda_1}{p-1-\delta_1}}L_2(d(x))^{\frac{\mu_1}{p-1-\delta_1}}\quad\text{ in }\Omega,
\ee
\be\label{psi2ex1}\psi_{2}(x)\,\sim \,d(x)^{\frac{q-k_2}{q-1-\delta_2}}L_2(d(x))^{\frac{\lambda_2}{q-1-\delta_2}}L_1(d(x))^{\frac{\mu_2}{q-1-\delta_2}}\quad\text{ in }\Omega.
\ee
In view of satisfying estimates given in \eqref{estsssolex1}, the comparison of the term $-\Delta_p\psi_{1}$ with $K_1(x){\psi_{1}}^{a_1}{\psi_{2}}^{b_1}$ on one side, and the term $-\Delta_q\psi_2$ with $K_2(x){\psi_2}^{a_2}{\psi_1}^{b_2}$ on the other side, imposes the exponents $\lambda_1,\lambda_2,\mu_1,\mu_2$ and $\delta_1,\delta_2$ to satisfy the following system:
\be\left\lbrace\begin{array}{l}
\delta_1\frac{p-k_1}{p-1-\delta_1}=a_1\frac{p-k_1}{p-1-\delta_1}+b_1\frac{q-k_2}{q-1-\delta_2},\vspace{0.2cm}\\
\delta_2\frac{q-k_2}{q-1-\delta_2}=a_2\frac{q-k_2}{q-1-\delta_2}+b_2\frac{p-k_1}{p-1-\delta_1},\vspace{0.2cm}\\
\lambda_1\frac{p-1}{p-1-\delta_1}=1+a_1\frac{\lambda_1}{p-1-\delta_1}+b_1\frac{\mu_2}{q-1-\delta_2},\vspace{0.2cm}\\
\lambda_2\frac{q-1}{q-1-\delta_2}=1+b_2\frac{\mu_1}{p-1-\delta_1}+a_2\frac{\lambda_2}{q-1-\delta_2},\vspace{0.2cm}\\
\mu_1\frac{p-1}{p-1-\delta_1}=a_1\frac{\mu_1}{p-1-\delta_1}+b_1\frac{\lambda_2}{q-1-\delta_2},\vspace{0.2cm}\\
\mu_2\frac{q-1}{q-1-\delta_2}=b_2\frac{\lambda_1}{p-1-\delta_1}+a_2\frac{\mu_2}{q-1-\delta_2}.\\
\end{array}\right.\nonumber\ee
Then, we get
\begin{small}
\be\label{exposant1}
\gamma_1=\frac{p-k_1}{p-1-\delta_1}=\frac{(p-k_1)(q-1-a_2)+(q-k_2)b_1}{(p-1-a_1)(q-1-a_2)-b_1b_2},
\ee

\be
\gamma_2=\frac{q-k_2}{q-1-\delta_2}= \frac{(q-k_2)(p-1-a_1)+(p-k_1)b_2}{(p-1-a_1)(q-1-a_2)-b_1b_2},
\ee

\be
\alpha_1=\frac{\lambda_1}{p-1-\delta_1}=\frac{q-1-a_2}{(p-1-a_1)(q-1-a_2)-b_1b_2},
\ee

\be
\alpha_2=\frac{\lambda_2}{q-1-\delta_2}=\frac{p-1-a_1}{(p-1-a_1)(q-1-a_2)-b_1b_2},
\ee

\be
\beta_1=\frac{\mu_1}{p-1-\delta_1}=\frac{b_1}{(p-1-a_1)(q-1-a_2)-b_1b_2},
\ee

\be\label{exposant6}
\beta_2=\frac{\mu_2}{q-1-\delta_2}=\frac{b_1}{(p-1-a_1)(q-1-a_2)-b_1b_2}, 
\ee
\end{small}
\\ 
{which imply} estimate \eqref{estsssolex1}. Moreover, inequalities \eqref{gamma} are then equivalent to inequalities \eqref{conddelta}.
Let $(u,v)\in[\underline{u},\overline{u}]\times[\underline{v},\overline{v}]$.
 { On one hand}, we have 
\be\label{subsolu1-0} -\Delta_p\underline{u}\leq m^{p-1}C_1L_1(d(x))^{\lambda_1+\delta_1\gamma_1}L_2(d(x))^{\mu_1+\delta_1\beta_1}d(x)^{\delta_1\gamma_1-k_1}\mbox{ in }\Omega,\nonumber\ee
\be\label{subsolv1-0}-\Delta_q\underline{v}\leq m^{\sigma(q-1)}C_1'L_2(d(x))^{\lambda_2+\delta_2\alpha_2}L_1(d(x))^{\mu_2+\delta_2\beta_2}d(x)^{\delta_2\gamma_2-k_2}\mbox{ in }\Omega.\nonumber\ee
{ On the other hand}, 
\be\label{subsolu2-0}
K_1(x)\underline{u}^{a_1}v^{b_1}\geq  C_2m^{a_1+\sigma\vert b_1\vert} \Lambda_1(d(x))d(x)^{a_1\gamma_1+b_1\gamma_2-k_1}\mbox{ in }\Omega
\nonumber\ee
with 
$\Lambda_1= {L_1}^{1+a_1\alpha_1+b_1\beta_2}.{L_2}^{a_1\beta_1+b_1\alpha_2}$.
Similarly, 
\be K_2(x)\underline{v}^{a_2}u^{b_2}\geq  C_2'm^{\sigma a_2+\vert b_2\vert} \Lambda_2(d(x))d(x)^{a_2\gamma_2+b_2\gamma_1-k_2}\mbox{ in }\Omega,
\nonumber\ee
with 
$\Lambda_2= {L_2}^{1+a_2\alpha_2+b_2\beta_1}.{L_1}^{a_2\beta_2+b_2\alpha_1}$.
Then, under condition \eqref{condexp1-0} and thanks to \eqref{exposant1} to \eqref{exposant6}, $(\underline{u},\underline{v})$ is a subsolutions pair of problem $\boldsymbol{({\rm P})}$, for $m$ small enough. Next,
\be\label{supersolu1-0} -\Delta_p\overline{u}\geq m^{1-p}{C_3}L_1(d(x))^{\lambda_1+\delta_1\gamma_1}L_2(d(x))^{\mu_1+\delta_1\beta_1}d(x)^{\delta_1\gamma_1-k_1}\mbox{ in }\Omega,\nonumber\ee
\be\label{supersolv1-0} -\Delta_q\overline{v}\geq m^{\sigma(1-q)}{C_3'}L_2(d(x))^{\lambda_2+\delta_2\alpha_2}L_1(d(x))^{\mu_2+\delta_2\beta_2}d(x)^{\delta_2\gamma_2-k_2}\mbox{ in }\Omega.\nonumber\ee
Furthermore,
\be\label{subsolu2-0}
K_1(x)\overline{u}^{a_1}v^{b_1}\leq C_4m^{-a_1-\sigma\vert b_1\vert} \Lambda_1(d(x))d(x)^{a_1\gamma_1+b_1\gamma_2-k_1}\mbox{ in }\Omega.
\nonumber\ee
Similarly,
\be
K_2(x)\overline{v}^{a_2}u^{b_2}\leq  C_4'm^{-\sigma a_2-\vert b_2\vert} \Lambda_2(d(x))d(x)^{a_2\gamma_2+b_2\gamma_1-k_2}\mbox{ in }\Omega.
\nonumber\ee
Then under \eqref{condexp1-0} and thanks to \eqref{exposant1} to \eqref{exposant6}, $(\overline{u},\overline{v})$ is a supersolutions pair of problem $\boldsymbol{({\rm P})}$, for $m$ small enough. Therefore estimates \eqref{soussol} to \eqref{sursol} hold. Let us check that conditions \eqref{f1} to \eqref{delta1} of Theorem \ref{TH1} are satisfied. By estimates \eqref{psi1ex1} and \eqref{psi2ex1} and using the properties of the perturbations $L_1$ and $L_2$ given in point (a) and (c) of Remark \ref{PropL}, for all $\varepsilon>0$ there exist positive constants $C_1,C_2$ and $C_1',C_2'$ such that 
\be C_1 d(x)^{\gamma_1}\leq \underline{u},\overline{u}\leq C_2 d(x)^{\gamma_1-\varepsilon}\quad\text{ and }\quad C_1' d(x)^{\gamma_2}\leq \underline{v},\overline{v}\leq C_2' d(x)^{\gamma_2-\varepsilon}\quad\text{ in }\Omega.\nonumber\ee
{ In addition}, using \eqref{exposant1} to \eqref{exposant6}, there exist positive constants $\kappa_1,$ $\kappa_2$ such that
\be \vert f_1(x,u,v)\vert= K_1(x)u^{a_1}v^{b_1}\leq \kappa_1 d(x)^{\delta_1\gamma_1-k_1-\varepsilon}\quad\text{ in }\Omega\times\mathcal{C},\nonumber\ee
\be \vert f_2(x,u,v)\vert= K_2(x)v^{a_2}u^{b_2}\leq \kappa_2 d(x)^{\delta_2\gamma_2-k_2-\varepsilon}\quad\text{ in }\Omega\times\mathcal{C}\nonumber\ee
and 
\be \left| \frac{\partial f_1}{\partial u}(x,u,v)\right|= \vert a_1\vert K_1(x)u^{a_1-1}v^{b_1}\leq \kappa_1 d(x)^{\left(\delta_1\gamma_1-k_1-\varepsilon\right)-\gamma_1}\quad\text{ in }\Omega\times\mathcal{C},\nonumber\ee
\be \left| \frac{\partial f_2}{\partial v}(x,u,v)\right|=\vert a_2\vert K_2(x)v^{a_2-1}u^{b_2}\leq \kappa_2 d(x)^{\left(\delta_2\gamma_2-k_2-\varepsilon\right)-\gamma_2}\quad\text{ in }\Omega\times\mathcal{C}.\nonumber\ee
Since $\gamma_1\in(1-\frac{1}{p},1)$ and $\gamma_2\in(1-\frac{1}{q},1)$, inequalities \eqref{delta1} hold for $\varepsilon$ small enough. 
Then, applying Theorem \ref{TH1} we conclude about the existence of positive solutions to $\boldsymbol{({\rm P})}$ in $\mathrm{W}^{1,p}_0(\Omega)\times\mathrm{W}^{1,q}_0(\Omega)$ satisfying the estimates \eqref{est1u} and \eqref{est1v}. 
\\\\
{ Finally, using Theorem \ref{regu}, we get that any positive weak solutions pair to $\boldsymbol{({\rm P})}$} in the conical shell $\mathcal{C}$ belongs to $\mathscr{C}^{0,\alpha}\left(\overline{\Omega}\right)\times\mathscr{C}^{0,\alpha}\left(\overline{\Omega}\right)$, for some $0<\alpha<1$. { This proves (i) of Theorem \ref{theo}}.

\textbf{\textit{Alternative 2:}} In this part, we look for positive solutions $(u,v)\in \mathrm{W}^{1,p}_0(\Omega)\times\mathrm{W}^{1,q}_0(\Omega)$ by making the "Ansatz" that both function $u$ and $v$ behave like the distance function $d(x)$ for $x\in \Omega$ near the boundary $\partial\Omega$. For that, similarly as in \textit{Alternative 1}, we take in \eqref{sssoluex1} and \eqref{sssolvex1}
\be\label{conddelta2} k_1-1<\delta_1< p-1\quad\text{ and }\quad\text k_2-1<\delta_2< q-1,\ee 
\be \mathcal{L}_1={L_1}\quad \text{ and }\quad\mathcal{L}_2={L_2}\quad\text{ in }\Omega.\nonumber\ee
By Theorem \ref{gms}, $\psi_1\in{\rm W}^{1,p}_0(\Omega)$, $\psi_2\in{\rm W}^{1,q}_0(\Omega)$ and satisfy
\be\label{psi1psi2ex1}\psi_{1}(x)\,\sim \,d(x)\quad\text{ and }\quad\psi_2(x)\,\sim\,d(x)\quad\text{ in }\Omega.
\nonumber\ee
In view of satisfying estimates given in \eqref{estsssolex1},  we fix $\delta_1$ and $\delta_2$ as follows:
\be \delta_1=a_1+b_1\quad\text{ and }\quad\delta_2=a_2+b_2.\ee
{ Then}, \eqref{estsssolex1} holds and inequalities given in \eqref{conddelta2}{ entail} \eqref{gamma2}. The rest of the proof is as in \textit{Alternative 1}. { This proves (ii) of Theorem \ref{theo}}.

\textbf{\textit{Alternative 3:}} Now we combine our methods from \textit{Alternative 1} and \textit{Alternative 2}. We search for  positive solutions $(u,v)\in{\rm W}^{1,p}_0(\Omega)\times{\rm W}^{1,q}_0(\Omega)$ to problem $\textbf{(P)}$ by again making the "Ansatz" that 
\be u(x)\,\sim\,d(x)^{\gamma}L_1(d(x))^{\alpha}L_2(d(x))^{\beta}\text{ in }\Omega,\nonumber\ee for some $\gamma\in(1-\frac{1}{p},1)$ and $\alpha,\beta\in\RR$, and $v$ behave like the distance function in $\Omega$. For that, we take in \eqref{sssoluex1} and \eqref{sssolvex1}
\be\label{conddelta3} k_1-2+\frac{k_1-1}{p-1}<\delta_1< k_1-1\quad\text{ and }\quad\text k_2-1<\delta_2< q-1,\ee 
\be \mathcal{L}_1={L_1}^{\lambda_1}.{L_2}^{\mu_1}\quad \text{ and }\quad\mathcal{L}_2={L_2}^{\lambda_2}.{L_1}^{\mu_2}\quad\text{ in }\Omega,\nonumber\ee
where $\lambda_1,\,\lambda_2,\,\mu_1,\,\mu_2\in \RR$ are suitable exponents {to be fixed}.
By Theorem \ref{gms}, $\psi_1\in{\rm W}^{1,p}_0(\Omega)$, $\psi_2\in{\rm W}^{1,q}_0(\Omega)$ and satisfy
\be\label{psi1ex1bis}\psi_{1}(x)\,\sim \,d(x)^{\frac{p-k_1}{p-1-\delta_1}}L_1(d(x))^{\frac{\lambda_1}{p-1-\delta_1}}L_2(d(x))^{\frac{\mu_1}{p-1-\delta_1}}\text{ and }
\psi_{2}(x)\,\sim \,d(x)\text{ in }\Omega.
\nonumber\ee
{ In view of } \eqref{estsssolex1}, the exponents have to satisfy
\be\left\lbrace\begin{array}{rclrcl}
\delta_1\frac{p-k_1}{p-1-\delta_1}&=&a_1\frac{p-k_1}{p-1-\delta_1}+b_1,\qquad &\delta_2&=&b_2\frac{p-k_1}{p-1-\delta_1}+a_2,\vspace{0.2cm}\\
\lambda_1\frac{p-1}{p-1-\delta_1}&=&a_1\frac{\lambda_1}{p-1-\delta_1}+1,\qquad&\lambda_2&=&b_2\frac{\mu_1}{p-1-\delta_1}+1,\vspace{0.2cm}\\
\mu_1\frac{p-1}{p-1-\delta_1}&=&a_1\frac{\mu_1}{p-1-\delta_1},\qquad& \mu_2&=&b_2\frac{\lambda_1}{p-1-\delta_1}.\\
\end{array}\right.\nonumber\ee
Hence we obtain
\begin{small}
\bea
&\gamma=\frac{p-k_1}{p-1-\delta_1}= \frac{p-k_1+b_1}{p-1-a_1}\quad\text{ and }\quad\delta_2=a_2+b_2\frac{p-k_1+b_1}{p-1-a_1},&\nonumber\\
\nonumber\\
&\alpha=\frac{\lambda_1}{p-1-\delta_1}=\frac{1}{p-1-\delta_1}\quad\text{ and }\quad\beta=\frac{\mu_1}{p-1-\delta_1}=0.&\nonumber
\eea
\end{small}
The rest of the proof  is as in \textit{Alternative 1}. { This proves (iii) of Theorem \ref{theo} and (iv) is the corresponding symmetric case of (iii)}.
\hfill{$\square$}

\subsubsection{Proof of Theorem \ref{uniq}}
To prove uniqueness of solutions, we apply a classical argument of    {\sc Kransnoselskii} \cite{Kra}. Let $(u,v),(\tilde{u},\tilde{v})\in{\rm W}^{1,p}_0(\Omega)\times{\rm W}^{1,q}_0(\Omega),$  be two distinct positive weak solutions pairs to problem $\boldsymbol{({\rm P})}$ in the conical shell $\mathcal{C}=[\underline{u},\overline{u}]\times[\underline{v},\overline{v}]$, where $(\underline{u},\underline{v})$, $(\overline{u},\overline{v})$ are given in the  proof of Theorem \ref{theo}. This means that ${T}(u,v)=(u,v)$ and ${T}(\tilde{u},\tilde{v})=(\tilde{u},\tilde{v})$, which implies that, $T_1\circ T_2(u)=u$, $T_2\circ T_1(v)=v$ and $T_1\circ T_2(\tilde{u})=\tilde{u}$, $T_1\circ T_2( \tilde{v})=\tilde{v}$, respectively. Let us define 
\be\label{cmax} C_{{\rm max}}\eqdef \sup \lbrace C\in\RR_+,\quad C\tilde{u}\leq u\quad\text{ and } \quad C\tilde{v}\leq v\quad\text{a.e. in }\Omega\rbrace.\ee
$$ T_1\circ T_2(C_{{\rm max}}\tilde{u})=(C_{{\rm max}})^{\frac{b_1}{p-1-a_1}.\frac{b_2}{q-1-a_2}}T_1\circ T_2(\tilde{u})=(C_{{\rm max}})^{\frac{b_1}{p-1-a_1}.\frac{b_2}{q-1-a_2}}\tilde{u},$$
$$ T_2\circ T_1(C_{{\rm max}}\tilde{v})=(C_{{\rm max}})^{\frac{b_2}{q-1-a_2}.\frac{b_1}{p-1-a_1}}T_2\circ T_1(\tilde{v})=(C_{{\rm max}})^{\frac{b_2}{q-1-a_2}.\frac{b_1}{p-1-a_2}}\tilde{v}.$$
Therefore, by Theorem \ref{wmp}, both mappings $T_1\circ T_2$ and $T_2\circ T_1$ being (pointwise) order-preserving, we arrive at 
\bea\label{cont} 
u=T_1\circ T_2(u)\geq T_1\circ T_2(C_{{\rm max}}\tilde{u})=(C_{{\rm max}})^{\frac{b_1}{p-1-a_1}.\frac{b_2}{q-1-a_2}}\tilde{u},\\
\nonumber\\
\label{contbis}v=T_2\circ T_1(v)\geq T_2\circ T_1(C_{{\rm max}}\tilde{v})=(C_{{\rm max}})^{\frac{b_2}{q-1-a_2}.\frac{b_1}{p-1-a_1}}\tilde{v}.
\eea
From $0<C_{{\rm max }}<1$ combined with the subhomogeneity condition \eqref{H1} we deduce that 
$$C'_{{\rm max}}\eqdef (C_{{\rm max }})^{\frac{b_1}{p-1-a_1}.\frac{b_2}{q-1-a_2}}> C_{{\rm max }},$$
which contradicts the maximality of the constant $C_{{\rm max}}$ in \eqref{cmax}, by inequalities \eqref{cont} and \eqref{contbis}. { Then, $C_{{\rm max}}\geq 1$ which entails $\tilde{u}\leq u$ and $\tilde{v}\leq v$ a.e. in $\Omega$. Interchanging the roles of $(u,v)$ and $(\tilde{u},\tilde{v}), $ we finally get $(u,v)= (\tilde{u},\tilde{v})$ a.e. in $\Omega$.}
\hfill{$\square$}\\\\
{\bf Proof of Theorem \ref{cor}} { The proof is very similar to the proof of Theorem \ref{theo}. So we omit it.}\hfill{$\square$}
\subsubsection{Proof of Theorem \ref{theo2}}
\textbf{\textit{Alternative 1:}} Assume that $a_1+b_1=k_1-1$ and $k_2-1\leq a_2+b_2< q-1$. 
We { look for positive sub and supersolutions pairs $(\underline{u},\underline{v}),$ $(\overline{u},\overline{v})$ in the form:} 
\be\underline{u}=m\psi_{1}\quad\text{ and }\quad \overline{u}=m^{-1}(\varphi_{1,p})^{1-\varepsilon}\quad\text{ in }\Omega,\nonumber\ee
\be\underline{v}=m^{\sigma}\psi_{2}\quad\text{ and }\quad\overline{v}=m^{-\sigma}(\varphi_{1,q})^{1-\sigma\varepsilon}\quad\text{ in }\Omega,\nonumber\ee
where $\sigma>0$ is given by \eqref{condexp1-0}, $\varepsilon<1$ and $m<1$ are appropriate positive constants small enough and $\psi_{1}\in {\rm W}^{1,p}_0(\Omega)$ and $\psi_{2}\in{\rm W}^{1,q}_0(\Omega)$ are the respective solutions to
$$-\Delta_p w=K_1(x)w^{\delta_1}\quad\text{ in }\Omega; \quad w|_{\partial\Omega}=0,\quad w>0\quad\text{ in }\Omega,$$
$$-\Delta_q w=K_2(x)w^{\delta_2}\quad\text{ in }\Omega; \quad w|_{\partial\Omega}=0,\quad w>0\quad\text{ in }\Omega,$$
with $k_1-1<\delta_1<p-1$ and $a_2+b_2<\delta_2<q-1$. By Theorem \ref{gms}, both $\psi_{1}$ and $\psi_{2}$ behave like the distance function in $\Omega$. Let us remark that by estimate \eqref{vep}, $\underline{u}\leq\overline{u}$ and $\underline{v}\leq\overline{v}$ in $\Omega$, for $m$ small enough.
Now, let $1<r<\infty$ and $\gamma\in(0,1)$, then we have
\begin{small}
\be
\begin{array}{rl}
-\Delta_r\left[\left(\varphi_{1,r}\right)^{\gamma}\right]&=\gamma^{r-1}\left[ \lambda_{1,r}(\varphi_{1,r})^{\gamma(r-1)}-(\gamma-1)(r-1)(\varphi_{1,r})^{(\gamma-1)(r-1)-1}\vert\nabla\varphi_{1,r}\vert^r\right]\\
&\\
&=\gamma^{r-1}(\varphi_{1,r})^{-(1-\gamma)(r-1)-1}\left[ \lambda_{1,r}(\varphi_{1,r})^{r}+(1-\gamma)(r-1)\vert\nabla\varphi_{1,r}\vert^r\right]
\end{array}
\nonumber\ee 
\end{small}
 in $\Omega$. By estimate \eqref{vep}, we conclude that 
\be\label{estpuissvep} -\Delta_r\left[\left(\varphi_{1,r}(x)\right)^{\gamma}\right]\sim d(x)^{-(1-\gamma)(r-1)-1}\quad\text{ in }\Omega. \ee
So, let $(u,v)\in[\underline{u},\overline{u}]\times[\underline{v},\overline{v}].$
 { On one hand}, we have 
\be -\Delta_p\underline{u}\leq m^{p-1}C_1 K_1(x)d(x)^{\delta_1}\quad\text{ and }\quad
-\Delta_q\underline{v}\leq m^{q-1}C_1' K_2(x)d(x)^{\delta_2}\quad\mbox{ in}\,\Omega.\nonumber\ee
On the other hand, we also have 
\be
\begin{array}{rcl} 
K_1(x)\underline{u}^{a_1}v^{b_1}&\geq &
\left\lbrace \begin{array}{lr} m^{a_1+\sigma b_1} K_1(x){\psi_1}^{a_1}{\psi_2}^{b_1}&\quad\text{ if }\; b_1> 0,\vspace{0.2cm}\\
m^{a_1-\sigma b_1} K_1(x){\psi_1}^{a_1}({\varphi_{1,q}})^{b_1(1-\varepsilon\sigma)}&\quad\text{ if }\; b_1< 0,
\\
\end{array}
\right.
\\~
\\
&\geq &\displaystyle  m^{a_1+\sigma\vert b_1\vert}C_2K_1(x)d(x)^{k_1-1+\varepsilon \sigma b_1^-}\quad\mbox{ in }\Omega,
\end{array}
\nonumber\ee
in $\Omega$. Similarly, we get
\be K_2(x) \underline{v}^{a_2}u^{b_2}\geq m^{\sigma a_2+\vert b_2\vert}C_2'K_2(x)d(x)^{a_2+b_2+\varepsilon b_2^-}\quad\mbox{ in }\Omega.\nonumber\ee
Then, for $m$ and $\varepsilon$ small enough, $(\underline{u},\underline{v})$ is a subsolutions pair of problem $\boldsymbol{({\rm P})}$.
Similarly, using estimate \eqref{estpuissvep}, we obtain
\be -\Delta_p \overline{u}\geq m^{1-p}C_3 d(x)^{-1-\varepsilon(p-1)}\;\text{ and }\;
-\Delta_q \overline{v}\geq m^{\sigma(1-q)}C_3' d(x)^{-1-\varepsilon\sigma(q-1)}\;\mbox{ in }\Omega.\nonumber\ee
Furthermore, by \eqref{ln}, for any $\varepsilon'>0$, there exists $C_4=C_4(\varepsilon')>0$  such that 
\be
\begin{array}{rcl} 
K_1(x)\overline{u}^{a_1}v^{b_1}&\leq &
\left\lbrace \begin{array}{lr}  m^{-(a_1+\sigma b_1)} K_1(x){(\varphi_{1,p})}^{a_1(1-\varepsilon)}({\varphi_{1,q}})^{b_1(1-\varepsilon\sigma)}&\text{ if }\; b_1> 0,\vspace{0.2cm}\\
 m^{-(a_1-\sigma b_1)} K_1(x){(\varphi_{1,p})}^{a_1(1-\varepsilon)}{\psi_2}^{b_1}&\text{ if }\; b_1< 0,
\\
\end{array}
\right.
\\~
\\
&\leq &\displaystyle  m^{-(a_1+\sigma\vert b_1\vert)}C_4d(x)^{-1-\varepsilon(a_1+ \sigma b_1^+)-\varepsilon'}\quad\mbox{ in }\Omega,
\end{array}
\nonumber\ee
Similarly, we have 
\be K_2(x) \underline{v}^{a_2}u^{b_2}\leq m^{-(\sigma a_2+\vert b_2\vert)}
C_4'd(x)^{-k_1+a_2+b_2-\varepsilon(\sigma a_2+ b_2^+)-\varepsilon'}\quad\mbox{ in } \Omega,\nonumber\ee
with $C_4'=C_4'(\varepsilon')$. Then, for $m$, $\varepsilon$ and $\varepsilon'$ small enough, $(\overline{u},\overline{v})$ is a supersolutions pair of problem $\boldsymbol{({\rm P})}$.
Applying Theorem \ref{TH1}, we get the existence of  positive  solutions $(u,v)\in{\rm W}^{1,p}_0(\Omega)\times{\rm W}^{1,q}_0(\Omega)$ of $\boldsymbol{({\rm P})}$ satisfying \eqref{estsol1-4}. 
This proves (i) of Theorem \ref{theo2}.

\textbf{\textit{Alternative 2:}} When $k_1-1\leq a_1+b_1 <q-1$ and $a_2+b_2=k_2-1$, interchanging the role of $u$ and $v$, the proof of (ii) is the same as above.

\textbf{\textit{Alternative 3:}} Assume that \eqref{garage} is satisfied. To prove \textit{(3)}, we follow the proof in \textit{Alternative 1}. We construct positive sub and supersolutions pairs $(\underline{u},\underline{v}),$ $(\overline{u},\overline{v})\in{\rm W}^{1,p}_0(\Omega)\times{\rm W}^{1,q}_0(\Omega)$ in the form 
\be\underline{u}=m(\varphi_{1,p})^{\gamma+\varepsilon},\; \overline{u}=m^{-1}(\varphi_{1,p})^{\gamma-\varepsilon}\text{ and }\underline{v}=m^{\sigma}\psi,\;\overline{v}=m^{-\sigma}(\varphi_{1,q})^{1-\sigma\varepsilon}\text{ in }\Omega,\nonumber\ee
where $\sigma>0$ is given by \eqref{condexp1-0}, and $\varepsilon, m$ are appropriate positive constants small enough and $\psi\in{\rm W}^{1,q}_0(\Omega)$ is the solution (see Theorem \ref{gms}) of
$$-\Delta_q w=K_2(x)w^{\delta}\quad\text{ in }\Omega; \quad w|_{\partial\Omega}=0,\quad w>0\quad\text{ in }\Omega,$$
with $a_2+\gamma b_2<\delta<q-1$. (iv) is the symmetric case of \textit{(3)}
 by interchanging the role of $u$ and $v$.  Finally, from Theorem \ref{regu}, we get the H\"older regularity of $(u,v)$.\hfill{$\square$}
\subsection{Example 2}\label{ex5}
We consider now the following singular system 
$$ \boldsymbol{({\rm P})}\left\lbrace
\begin{array}{ll}
          -\Delta_pu=u^{a_1}v^{b_1}-u^{\alpha_1}v^{\beta_1}\quad\text{in }\Omega\, ;\quad u|_{\partial\Omega}=0,\quad u>0\quad\text{in }\Omega,\vspace{0.2cm}\\
          -\Delta_qv=v^{a_2}u^{b_2}-v^{\alpha_2}u^{\beta_2}\quad\text{in }\Omega\, ;\quad v|_{\partial\Omega}=0,\quad v>0\quad\text{in }\Omega,
 \end{array}
\right.
$$
where the above exponents satisfy
\be\label{condexp1-4}(p-1-a_1)-\sigma\vert b_1\vert>0\quad\text{ and }\quad(\alpha_1-a_1)-\sigma(\vert\beta_1\vert-\vert b_1\vert)>0,\ee
\be\label{condexp2-4} \sigma(q-1-a_2)-\vert b_2\vert>0\quad\text{ and }\quad \sigma(\alpha_2-a_2)-(\vert\beta_2\vert-\vert b_2\vert)>0,\ee
for some constant $\sigma>0$. Then, we have the following result:
\begin{theorem}~
\begin{enumerate}
\item { Let}
{\footnotesize\be\label{gamma12-4} \gamma_1=\frac{p(q-1-a_2)+qb_1}{(p-1-a_1)(q-1-a_2)-b_1b_2},\;\; \gamma_2=\frac{q(p-1-a_1)+pb_2}{(p-1-a_1)(q-1-a_2)-b_1b_2}	\ee}
and assume that 
\be\label{condexp3-4} 1-\frac{1}{p}<\gamma_1<1\quad\text{ and }\quad (\alpha_1-a_1)\gamma_1+(\beta_1-b_1)\gamma_2>0,\ee
\be\label{condexp4-4} 1-\frac{1}{q}<\gamma_2<1\quad\text{ and }\quad (\alpha_2-a_2)\gamma_2+(\beta_2-b_2)\gamma_1>0.\ee
Then, problem $\boldsymbol{({\rm P})}$ has a positive solution $(u,v)\in{\rm W}^{1,p}_0(\Omega)\times{\rm W}^{1,q}_0(\Omega)$ { satisfying}
\be\label{estsol1-4} u(x)\sim d(x)^{\gamma_1}\quad\text{ and }\quad v(x)\sim d(x)^{\gamma_2}\quad\text{ in }\Omega.\ee
In addition, we have $(u,v)\in\mathscr{C}^{0,\alpha}\left(\overline{\Omega}\right)\times\mathscr{C}^{0,\alpha}\left(\overline{\Omega}\right),$ for some $0<\alpha<1$.
\item Assume that 
\be\label{condexp5-4} -1<a_1+b_1<p-1\quad\text{ and }\quad (\alpha_1-a_1)+(\beta_1-b_1)>0,
\ee
\be\label{condexp6-4} -1<a_2+b_2<q-1\quad\text{ and }\quad (\alpha_2-a_2)+(\beta_2-b_2)>0.
\ee
Then, $\boldsymbol{({\rm P})}$ has a positive  solution $(u,v)\in{\rm W}^{1,p}_0(\Omega)\times{\rm W}^{1,q}_0(\Omega)$ satisfying
\be\label{estsol2-4} u(x)\sim d(x)\quad\text{ and }\quad v(x)\sim d(x)\quad\text{ in }\Omega.\ee
In addition, we have $(u,v)\in\mathscr{C}^{1,\alpha}\left(\overline{\Omega}\right)\times\mathscr{C}^{1,\alpha}\left(\overline{\Omega}\right),$ for some $0<\alpha<1$.
\item 
Let
\be\label{gamma3-4} \gamma=\frac{p+b_1}{p-1-a_1}\ee and assume that 
\be\label{condexp7-4} 1-\frac{1}{p}<\gamma<1\quad  \text{ and }\quad (\alpha_1-a_1)\gamma+(\beta_1-b_1)>0,\ee
\be\label{condexp8-4} -1<a_2+{b_2}{\gamma}<p-1 \quad\text{ and }\quad (\alpha_2-a_2)+(\beta_2-b_2)\gamma>0.\ee
Then, $\boldsymbol{({\rm P})}$ has a  positive solution  $(u,v)\in{\rm W}^{1,p}_0(\Omega)\times{\rm W}^{1,q}_0(\Omega)$ satisfying
\be\label{estsol3-4} u(x)\sim d(x)^{\gamma}\quad\text{ and }\quad v(x)\sim d(x)\quad\text{ in }\Omega.\ee
In addition, we have $(u,v)\in\mathscr{C}^{0,\alpha}\left(\overline{\Omega}\right)\times\mathscr{C}^{1,\alpha}\left(\overline{\Omega}\right),$ for some $0<\alpha<1$.
\item Symmetrically, set
\be\label{gamma4-4} \gamma=\frac{q+b_2}{q-1-a_2}\ee and assume that 
\be\label{condexp9-4}  -1<a_1+{b_1}{\gamma}<p-1\quad  \text{ and }\quad(\alpha_1-a_1)+(\beta_1-b_1)\gamma>0,\ee
\be\label{condexp10-4} 1-\frac{1}{q}<\gamma<1\quad\text{ and }\quad (\alpha_2-a_2)\gamma+(\beta_2-b_2)>0.\ee
Then, $\boldsymbol{({\rm P})}$ has a  positive  solution $(u,v)\in{\rm W}^{1,p}_0(\Omega)\times{\rm W}^{1,q}_0(\Omega)$ satisfying
\be\label{estsol4-4} u(x)\sim d(x)\quad\text{ and }\quad v(x)\sim d(x)^{\gamma}\quad\text{ in }\Omega.\ee
In addition, we have $(u,v)\in\mathscr{C}^{1,\alpha}\left(\overline{\Omega}\right)\times\mathscr{C}^{0,\alpha}\left(\overline{\Omega}\right),$ for some $0<\alpha<1$.
\end{enumerate}
\end{theorem}

\begin{proof}
We apply Theorem \ref{TH1} with
\be \underline{u}\equiv m \psi_{1},\quad \overline{u}\equiv m^{-1}\psi_{1}\quad \text{ and }\quad\underline{v}\equiv m^{\sigma} \psi_{2},\quad \overline{v}\equiv m^{-\sigma}\psi_{2}\quad\text{ in }\Omega,\nonumber\ee
where $\sigma>0$ is the constant given in \eqref{condexp1-4} and \eqref{condexp2-4}, $m<1$ is a positive constant small enough and  $\psi_{1}\in{\rm W}^{1,p}_0(\Omega)$, $\psi_{2}\in{\rm W}^{1,q}_0(\Omega)$ are given by Theorem \ref{gms} as the respective unique solutions of problems
\be -\Delta_p w=w^{\delta_1}\quad\text{ in }\Omega;\quad w\vert_{\partial\Omega}=0,\quad w>0\quad \text{ in }\Omega,\nonumber\ee
\be -\Delta_q w=w^{\delta_2}\quad\text{ in }\Omega;\quad w\vert_{\partial\Omega}=0,\quad w>0\quad \text{ in }\Omega,\nonumber\ee
satisfying some cone conditions we precise below. In the following \textit{Alternatives}, we choose $-2-\frac{1}{p-1}<\delta_1<p-1$ and $-2-\frac{1}{q-1}<\delta_2<q-1$ such that
\be\label{est-sssol-4}  -\Delta_p\psi_1\sim{ \psi_1}^{a_1}{\psi_2}^{b_1}\quad\text{ and }\quad-\Delta_q\psi_2\sim{ \psi_2}^{a_2}{\psi_1}^{b_2}\quad\text{ in }\Omega.\ee

\textbf{\textit{Alternative 1:}} Assume that conditions \eqref{condexp3-4} and \eqref{condexp4-4} hold. Then, arguing as in \textit{Alternative 1} in the proof of Theorem \ref{theo}, we choose $-2-\frac{1}{p-1}<\delta_1<-1$ and $-2-\frac{1}{q-1}<\delta_2<-1$ unique solutions pair of the following system:
\begin{small}
\be
\frac{ \delta_1 p}{p-1-\delta_1}=\frac{a_1 p}{p-1-\delta_1}+\frac{b_1 q}{q-1-\delta_2}\quad\text{ and }\quad
\frac{\delta_2 q}{q-1-\delta_2}=\frac{a_2 q}{q-1-\delta_2}+\frac{b_2 p}{p-1-\delta_2}.
\nonumber\ee
\end{small}
Since
\be \psi_1(x)\sim d(x)^{\gamma_1}\quad\text{ and }\quad\psi_2(x)\sim d(x)^{\gamma_2}\quad\text{ in }\Omega,\nonumber\ee 
where $\gamma_1=\frac{p}{p-1-\delta_1}$ and $\gamma_2=\frac{q}{q-1-\delta_2}$ are given by \eqref{gamma12-4}, estimates \eqref{est-sssol-4} {follows}.
Let $(u,v)\in[\underline{u},\overline{u}]\times[\underline{v},\overline{v}]$.
First, we have 
\be\label{subsolu1-4}
 -\Delta_p\underline{u}\leq m^{p-1}C_1d(x)^{\delta_1\gamma_1}\quad\text{ and }\quad-\Delta_q\underline{v}\leq m^{\sigma(q-1)}C_1'd(x)^{\delta_2\gamma_2}\quad\mbox{ in }\Omega.\ee
On the other hand, by \eqref{condexp1-4} and \eqref{condexp3-4},
\bea\label{subsolu2-4}
\underline{u}^{a_1}v^{b_1}-\underline{u}^{\alpha_1}v^{\beta_1}&\geq & m^{a_1+\sigma\vert b_1\vert}{\psi_1}^{a_1}{\psi_2}^{b_1}
\left[1-m^{\alpha_1-a_1-\sigma(\vert \beta_1\vert-\vert b_1\vert)}{\psi_1}^{\alpha_1-a_1}{\psi_2}^{\beta_1-b_1}\right]\nonumber\vspace{0.2cm}\\
&\geq&  m^{a_1+\sigma\vert b_1\vert} C_2d(x)^{a_1\gamma_1+b_1\gamma_2}\quad\mbox{ in }\Omega.
\eea
 for $m$ small enough. 
By \eqref{condexp2-4} and \eqref{condexp4-4}, we also have 
\be\label{subsolv2-4}\underline{v}^{a_2}u^{b_2}-\underline{v}^{\alpha_2}u^{\beta_2}\geq m^{\sigma a_2+\vert b_2\vert} C_2'd(x)^{a_2\gamma_2+b_2\gamma_1}\quad\mbox{ in }\Omega,\ee
for $m$ small enough. Then, under conditions \eqref{condexp1-4}, \eqref{condexp2-4}, \eqref{condexp3-4} and \eqref{condexp4-4} and for $m$ small enough, $(\underline{u},\underline{v})$ is a subsolutions pair of problem $\boldsymbol{({\rm P})}$.

Similarly, we have 
\be\label{supersolu1-4} 
-\Delta_p\overline{u}\geq m^{1-p}{C_3}d(x)^{\delta_1\gamma_1}\quad\text{ and }\quad -\Delta_q\overline{v}\geq m^{\sigma(1-q)}{C_3'}d(x)^{\delta_2\gamma_2}\quad\mbox{ in }\Omega.\ee
In addition,
\be\label{supersolu2-4}
\overline{u}^{a_1}v^{b_1}-\overline{u}^{\alpha_1}v^{\beta_1}\leq m^{-a_1-\sigma\vert b_1\vert}{\psi_1}^{a_1}{\psi_2}^{b_1}\leq  m^{-a_1-\sigma\vert b_1\vert}C_4d(x)^{a_1\gamma_1+b_1\gamma_2}
\ee
in $\Omega$. We obtain further
\be\label{supersolv2-4}
\overline{v}^{a_2}u^{b_2}-\overline{v}^{\alpha_2}u^{\beta_2}\leq  m^{-\sigma a_2-\vert b_2\vert}C_4'd(x)^{a_2\gamma_2+b_2\gamma_1}\quad\mbox{ in }\Omega.\ee
 Then, under conditions \eqref{condexp1-4}, \eqref{condexp2-4} and for $m$ small enough, $(\underline{u},\underline{v})$ is a supersolutions pair of problem $\boldsymbol{({\rm P})}$.

Applying Theorem \ref{TH1}, we get the existence of  positive  solutions $(u,v)\in{\rm W}^{1,p}_0(\Omega)\times{\rm W}^{1,q}_0(\Omega)$ of $\boldsymbol{({\rm P})}$ satisfying \eqref{estsol1-4}. Again from Theorem  \ref{regu}, $(u,v)$ are H\"older continuous. This proves the assertion (i).

\textbf{\textit{Alternative 2:}} Now, assume that conditions \eqref{condexp5-4} and \eqref{condexp6-4} are satisfied. Then, we choose $\delta_1=a_1+b_1$ and $\delta_2=a_2+b_2$. By Theorem \ref{gms}, since \be \psi_1(x)\sim d(x)\quad\text{ and }\quad\psi_2(x)\sim d(x)\quad\text{ in }\Omega,\nonumber\ee
estimates \eqref{est-sssol-4} hold.
 Instead of inequalities \eqref{subsolu1-4}, we have in this case 
\be -\Delta_p\underline{u}\leq  m^{p-1} C_1d(x)^{a_1+b_1}\;\text{ and }\; -\Delta_q\underline{v}\leq  m^{\sigma(q-1)} C_1' d(x)^{a_2+b_2}\quad\text{ in }\Omega.\nonumber\ee
From \eqref{condexp1-4}, \eqref{condexp2-4}, \eqref{condexp5-4} and \eqref{condexp6-4}, we get for any $(u,v)\in [\underline{u},\overline{u}]\times[\underline{v},\overline{v}]$:
\be \underline{u}^{a_1}v^{b_1}-\underline{u}^{\alpha_1}v^{\beta_1}\geq m^{a_1+\sigma\vert b_1\vert} C_2d(x)^{a_1+b_1}\quad\text{ in }\Omega,\nonumber\ee
\be \underline{v}^{a_2}u^{b_2}-\underline{v}^{\alpha_2}u^{\beta_2}\geq m^{\sigma a_2+\vert b_2\vert} C_2'd(x)^{a_2+b_2} \quad\text{ in }\Omega,\nonumber\ee
for $m$ small enough.
Then, under conditions \eqref{condexp1-4}, \eqref{condexp2-4}, \eqref{condexp5-4}, \eqref{condexp6-4} and for $m$ small enough, $(\underline{u},\underline{v})$ is a subsolution pair of problem $\boldsymbol{({\rm P})}$. 
 Instead of inequalities \eqref{supersolu1-4}, we have in this case in $\Omega$,
\be -\Delta_p\overline{u}\geq m^{1-p}C_3d(x)^{a_1+b_1}\quad\text{ and }\quad -\Delta_q\overline{v}\geq m^{\sigma(1-q)}C_3'd(x)^{a_2+b_2}.\nonumber\ee
In addition, instead of inequalities \eqref{supersolu2-4} and \eqref{supersolv2-4}, we get
\be
\overline{u}^{a_1}v^{b_1}-\overline{u}^{\alpha_1}v^{\beta_1}\leq  m^{-a_1-\sigma\vert b_1\vert}C_4d(x)^{a_1+b_1},
\nonumber\ee
\be
\overline{v}^{a_2}u^{b_2}-\overline{v}^{\alpha_2}u^{\beta_2}\leq  m^{-\sigma a_2-\vert b_2\vert}C_4'd(x)^{a_2+b_2},
\nonumber\ee
in $\Omega$. Then, under conditions \eqref{condexp1-4}, \eqref{condexp2-4} and for $m $ small enough, $(\overline{u},\overline{v})$ is a supersolution pair of problem $\boldsymbol{({\rm P})}$. 
Then, we conclude as in the \textit{{Alternative 1}} and (ii) is proved.

\textbf{\textit{Alternative 3:}} Now, assume conditions \eqref{condexp7-4} and \eqref{condexp8-4} hold. Then, arguing as in the proof of Theorem \ref{theo}, we choose $-2-\frac{1}{p}<\delta_1<-1$ and $-1<\delta_2<q-1$ unique solutions pair of the following {system}:
\be
\frac{\delta_1 p}{p-1-\delta_1}=\frac{a_1p}{p-1-\delta_1}+b_1\quad\text{ and }\quad \delta_2=a_2+\frac{b_2p}{p-1-\delta_2}.
\nonumber\ee
Estimates in \eqref{est-sssol-4} hold since
\be \psi_1(x)\sim d(x)^{\gamma}\quad\text{ and }\quad\psi_2(x)\sim d(x)\quad\text{ in }\Omega,\nonumber\ee 
{ with $\gamma$ given by \eqref{gamma3-4}}.
Instead of inequalities \eqref{subsolu1-4}, we have in this case 
\be -\Delta_p\underline{u}\leq  m^{p-1} C_1 d(x)^{\delta_1\gamma}\quad\text{ and }\quad-\Delta_q\underline{v}\leq  m^{\sigma(q-1)} C_1' d(x)^{\delta_2}\quad\text{ in }\Omega.\nonumber\ee
From \eqref{condexp1-4}, \eqref{condexp2-4}, \eqref{condexp7-4} and \eqref{condexp8-4}, we obtain now
\be \underline{u}^{a_1}v^{b_1}-\underline{u}^{\alpha_1}v^{\beta_1}\geq m^{a_1+\sigma\vert b_2\vert}C_2 d(x)^{a_1\gamma+b_1}\quad\mbox{ in }\Omega,\nonumber\ee
\be \underline{v}^{a_2}u^{b_2}-\underline{v}^{\alpha_2}u^{\beta_2}\geq m^{\sigma a_2+\vert b_2\vert}C_2' d(x)^{a_2+b_2\gamma} \quad\mbox{ in }\Omega,\nonumber\ee
for $m$ small enough.
Then, under conditions \eqref{condexp1-4}, \eqref{condexp2-4}, \eqref{condexp7-4}, \eqref{condexp8-4} and for $m $ small enough, $(\underline{u},\underline{v})$ is a subsolution pair of problem $\boldsymbol{({\rm P})}$. 
Instead of \eqref{supersolu1-4}, we have  
\be -\Delta_p\overline{u}\geq m^{1-p}C_3d(x)^{\delta_1\gamma}\quad\text{ and }\quad-\Delta_p\overline{v}\geq m^{\sigma(1-q)}C_3'd(x)^{\delta_2}\quad\mbox{ in}\,\Omega.\nonumber\ee
And inequalities \eqref{supersolu2-4} are replaced by
\be
\overline{u}^{a_1}v^{b_1}-\overline{u}^{\alpha_1}v^{\beta_1}\leq  m^{-a_1-\sigma\vert b_1\vert}C_4d(x)^{a_1\gamma+b_1}\quad\mbox{ in }\Omega,
\nonumber\ee
\be
\overline{v}^{a_2}u^{b_2}-\overline{v}^{\alpha_2}u^{\beta_2}\leq  m^{-\sigma a_2-\vert b_2\vert}C_4'd(x)^{a_2+b_2\gamma}\quad\mbox{ in }\Omega.
\nonumber\ee
Then, under conditions \eqref{condexp1-4}, \eqref{condexp2-4} and for $m $ small enough, $(\overline{u},\overline{v})$ is a supersolution pair of problem $\boldsymbol{({\rm P})}$. 
We conclude as in the \textit{{Alternative 1}}. Thus, (iii) is proved. Note that (iv) is the symmetric case of (iii)  by interchanging $u$ and $v$.
\end{proof}
{ We can further prove similarly (we omit the proof)}:
\begin{theorem}
Assume that conditions \eqref{condexp1-4} and \eqref{condexp2-4} are satisfied.
\begin{enumerate}
\item Assume that \be a_1+b_1=-1\quad\text{ and }\quad (\alpha_1-a_1)+(\beta_1-b_1)>0,\ee
\be -1\leq a_2+b_2< q-1 \quad\text{ and }\quad (\alpha_2-a_2)+(\beta_2-b_2)>0.\ee 
Then, for all $\varepsilon>0$ small enough, there exist  $C_1,C_2>0$ and $C_1',C_2'>0$ such that $\mathrm{\boldsymbol{(P)}}$ admits  positive  solutions $(u,v)\in{\rm W}^{1,p}_0(\Omega)\times{\rm W}^{1,q}_0(\Omega)$ satisfying:
\be C_1d(x)\leq u \leq C_2 d(x)^{1-\varepsilon}\text{ and } C_1'd(x) \leq v\leq C_2' d(x)^{1-\varepsilon\sigma}\text{ in }\Omega,\ee
with $\sigma>0$ is given in \eqref{condexp1-0}. In addition, we have $(u,v)\in \mathscr{C}^{0,\alpha}\left(\overline{\Omega}\right)\times\mathscr{C}^{0,\alpha}\left(\overline{\Omega}\right)$, for some $0<\alpha<1$.
\item Symmetrically, assume that 
\be -1\leq a_1+b_1<q-1\quad\text{ and }\quad(\alpha_1-a_1)+(\beta_1-b_1)>0,\ee 
\be a_2+b_2=-1\quad\text{ and }\quad (\alpha_2-a_2)+(\beta_2-b_2)>0.\ee
Then, for all $\varepsilon>0$ small enough, there exist  $C_1,C_2>0$ and $C_1',C_2'>0$ such that $\mathrm{\boldsymbol{(P)}}$ admits  positive  solutions $(u,v)\in{\rm W}^{1,p}_0(\Omega)\times{\rm W}^{1,q}_0(\Omega)$ satisfying:
\be C_1d(x)\leq u \leq C_2 d(x)^{1-\varepsilon}\text{ and } C_1'd(x) \leq v\leq C_2' d(x)^{1-\varepsilon\sigma}\text{ in }\Omega.\ee
In addition, we have $(u,v)\in \mathscr{C}^{0,\alpha}\left(\overline{\Omega}\right)\times\mathscr{C}^{0,\alpha}\left(\overline{\Omega}\right)$, for some $0<\alpha<1$.
\item Let
\be \gamma=\frac{p+b_1}{p-1-a_1}\nonumber\ee
and assume that 
\be
1-\frac{1}{p}<\gamma<1\quad\;and\;\quad (\alpha_1-a_1)\gamma+(\beta_1-b_1)>0,\ee
\be a_2+b_2\gamma=-1\quad\text{ and }\quad(\alpha_2-a_2)+(\beta_2-b_2)\gamma>0.
\ee
Then, for all $\varepsilon>0$ small enough, there exist  $C_1,C_2>0$ and $C_1',C_2'>0$ such that $\mathrm{\boldsymbol{(P)}}$ admits  positive  solutions $(u,v)\in{\rm W}^{1,p}_0(\Omega)\times{\rm W}^{1,q}_0(\Omega)$ satisfying:
{\small \be C_1d(x)^{\gamma+\varepsilon}\leq u \leq C_2 d(x)^{\gamma-\varepsilon}\text{ and } C_1'd(x) \leq v\leq C_2' d(x)^{1-\varepsilon\sigma}\text{ in }\Omega.\ee}
In addition, we have $(u,v)\in \mathscr{C}^{0,\alpha}\left(\overline{\Omega}\right)\times\mathscr{C}^{0,\alpha}\left(\overline{\Omega}\right)$, for some $0<\alpha<1$.
\item Symmetrically, let
\be \gamma=\frac{q+b_2}{q-1-a_2}\nonumber\ee
and assume that 
\be\label{garage2}
a_1+b_1\gamma=-1\quad\;and\;\quad (\alpha_1-a_1)+(\beta_1-b_1)\gamma>0,\ee
\be 1-\frac{1}{q}<\gamma<1 \quad\text{ and }\quad (\alpha_2-a_2)\gamma+(\beta_2-b_2)>0.
\ee
Then, for all $\varepsilon>0$ small enough, there exist  $C_1,C_2>0$ and $C_1',C_2'>0$ such that $\mathrm{\boldsymbol{(P)}}$ admits  positive  solutions $(u,v)\in{\rm W}^{1,p}_0(\Omega)\times{\rm W}^{1,q}_0(\Omega)$ satisfying:
{\small \be C_1d(x)\leq u \leq C_2 d(x)^{1-\varepsilon}\text{ and }C_1'd(x)^{\gamma+\varepsilon\sigma} \leq v\leq C_2' d(x)^{\gamma-\varepsilon\sigma}\text{ in }\Omega.\ee}
In addition, we have $(u,v)\in \mathscr{C}^{0,\alpha}\left(\overline{\Omega}\right)\times\mathscr{C}^{0,\alpha}\left(\overline{\Omega}\right)$, for some $0<\alpha<1$.
\end{enumerate}
\end{theorem}

\subsection{Example 3}\label{ex2}
In this section, we consider the following singular competition system 
$$ \boldsymbol{({\rm P})}\left\lbrace
\begin{array}{ll}
          -\Delta_pu=\lambda_1 u^{\alpha_1}-u^{\beta_1}-\mu_1 u^{a_1}v^{b_1}\quad\text{in }\Omega\, ;\quad u|_{\partial\Omega}=0,\quad u>0\quad\text{in }\Omega,\vspace{0.2cm}\\
          -\Delta_qv=\lambda_2 v^{\alpha_2}-v^{\beta_2}-\mu_2 v^{a_2}u^{b_2}\quad\text{in }\Omega\, ;\quad v|_{\partial\Omega}=0,\quad v>0\quad\text{in }\Omega,
 \end{array}
\right.
$$

where 
$\lambda_1,\lambda_2$ and $\mu_1,\mu_2$ are positive and $\alpha_1,\alpha_2,\beta_1,\beta_2, a_1,a_2, b_1,b_2$ satisfy
\be\label{condexp1} -2-\frac{1}{p-1}<\alpha_1<p-1,\quad\alpha_1<\beta_1\quad\text{ and }\quad a_1-\alpha_1-\sigma\vert b_1\vert >0,\ee
\be\label{condexp2} -2-\frac{1}{q-1}<\alpha_2<q-1,\quad\alpha_2<\beta_2\quad\text{ and }\quad \sigma(a_2-\alpha_2)-\vert b_2\vert >0,\ee
for some constant $\sigma>0$. { Then, we have}
\begin{theorem}\label{ex3}
\begin{enumerate}
\item 
Assume that 
\be\label{condexp3} -2-\frac{1}{p-1}<\alpha_1<-1\quad \text{ and }\quad\frac{(a_1-\alpha_1)p}{p-1-\alpha_1}+\frac{b_1q}{q-1-\alpha_2}>0,\ee
\be\label{condexp4} -2-\frac{1}{q-1}<\alpha_2<-1\quad \text{ and }\quad\frac{(a_2-\alpha_2)q}{q-1-\alpha_2}+\frac{b_2p}{p-1-\alpha_1}>0.\ee
Then, $\boldsymbol{({\rm P})}$ admits  positive solutions $(u,v)\in{\rm W}^{1,p}_0(\Omega)\times{\rm W}^{1,q}_0(\Omega)$ satisfying:
\be\label{estsol1} u(x)\sim d(x)^{\frac{p}{p-1-\alpha_1}}\quad\text{ and }\quad v(x)\sim d(x)^{\frac{q}{q-1-\alpha_2}}\quad\text{ in }\Omega.\ee
In addition, we have $(u,v)\in\mathscr{C}^{0,\alpha}\left(\overline{\Omega}\right)\times\mathscr{C}^{0,\alpha}\left(\overline{\Omega}\right),$ for some $0<\alpha<1$.
\item Assume that 
\be\label{condexp5} -1<\alpha_1<p-1\quad\text{ and }\quad a_1-\alpha_1+b_1>0,\ee
\be\label{condexp6} -1<\alpha_2<q-1\quad\text{ and }\quad a_2-\alpha_2+b_2>0.\ee
Then, $\boldsymbol{({\rm P})}$ admits positive solutions $(u,v)\in{\rm W}^{1,p}_0(\Omega)\times{\rm W}^{1,q}_0(\Omega)$ satisfying:
\be\label{estsol3} u(x)\sim d(x)\quad\text{ and }\quad v(x)\sim d(x)\quad\text{ in }\Omega.\ee
In addition, we have $(u,v)\in\mathscr{C}^{1,\alpha}\left(\overline{\Omega}\right)\times\mathscr{C}^{1,\alpha}\left(\overline{\Omega}\right),$ for some $0<\alpha<1$.
\item Assume that 
\be\label{condexp7} -2-\frac{1}{p-1}<\alpha_1<-1\quad \text{ and }\quad (a_1-\alpha_1+b_1)p-b_1(\alpha_1+1)>0  ,\ee
\be\label{condexp8} -1<\alpha_2<q-1\quad \text{ and }\quad (a_2-\alpha_2+b_2)p-(a_2-\alpha_2)(\alpha_1+1)>0.\ee
Then, $\boldsymbol{({\rm P})}$ admits positive solutions $(u,v)\in{\rm W}^{1,p}_0(\Omega)\times{\rm W}^{1,q}_0(\Omega)$ satisfying:
\be\label{estsol5} u(x)\sim d(x)^{\frac{p}{p-1-\alpha_1}}\quad\text{ and }\quad v(x)\sim d(x)\quad\text{ in }\Omega.\ee
In addition, we have $(u,v)\in\mathscr{C}^{0,\alpha}\left(\overline{\Omega}\right)\times\mathscr{C}^{1,\alpha}\left(\overline{\Omega}\right),$ for some $0<\alpha<1$.
\item Symmetrically, assume that 
\be\label{condexp9} -1<\alpha_1<p-1\quad  \text{ and }\quad (a_1-\alpha_1+b_1)q-(a_1-\alpha_1)(\alpha_2+1)>0,\ee
\be\label{condexp10} -2-\frac{1}{q-1}<\alpha_2<-1\quad \text{ and }\quad (a_2-\alpha_2+b_2)q-b_2(\alpha_2+1)>0.\ee
Then, $\boldsymbol{({\rm P})}$ admits positive solutions $(u,v)\in{\rm W}^{1,p}_0(\Omega)\times{\rm W}^{1,q}_0(\Omega)$ satisfying:
\be\label{estsol7} u(x)\sim d(x)\quad\text{ and }\quad v(x)\sim d(x)^{\frac{q}{q-1-\alpha_2}}\quad\text{ in }\Omega.\ee
In addition, we have $(u,v)\in\mathscr{C}^{1,\alpha}\left(\overline{\Omega}\right)\times\mathscr{C}^{0,\alpha}\left(\overline{\Omega}\right),$ for some $0<\alpha<1$.
\end{enumerate}
\end{theorem}
\begin{proof}
 We apply Theorem \ref{TH1} with 
\be\label{psi1} \underline{u}\equiv m \psi_{1},\quad \overline{u}\equiv m^{-1}\psi_{1}\quad\text{ and }\quad \underline{v}\equiv m^{\sigma} \psi_{2},\quad \overline{v}\equiv m^{-\sigma}\psi_{2}\quad\text{ in }\Omega,\ee
where $\sigma>0$ is the constant given in \eqref{condexp1} and \eqref{condexp2}, $m<1$ is a suitable small positive constant and  $\psi_{1}\in{\rm W}^{1,p}_0(\Omega)$, $\psi_{2}\in{\rm W}^{1,q}_0(\Omega)$ are (given by Theorem \ref{gms}) the respective unique solutions of problems
\be -\Delta_p w=w^{\alpha_1}\quad\text{ in }\Omega;\quad w\vert_{\partial\Omega}=0,\quad w>0\quad \text{ in }\Omega,\nonumber\ee
\be -\Delta_q w=w^{\alpha_2}\quad\text{ in }\Omega;\quad w\vert_{\partial\Omega}=0,\quad w>0\quad \text{ in }\Omega.\nonumber\ee

\textit{\textbf{Alternative 1}}: Assume conditions \eqref{condexp3} and \eqref{condexp4} are satisfied. Then, from Theorem \ref{gms}, we get
\be\label{estsssol1} \psi_{1}(x)\sim d(x)^{\frac{p}{p-1-\alpha_1}}\quad \text{ and }\quad \psi_{2}(x)\sim d(x)^{\frac{q}{q-1-\alpha_2}}\quad \text{ in }\Omega.\nonumber\ee
Let us prove that, for $m$ small enough, $(\underline{u},\underline{v})$ and $(\overline{u},\overline{v})$ are respectively sub and supersolutions pairs of $\boldsymbol{({\rm P})}$.
Let $(u,v)\in[\underline{u},\overline{u}]\times\left[\underline{v},\overline{v}\right]$.
We have in $\Omega$,
\be\label{subsolu11} -\Delta_p\underline{u}\leq  m^{p-1} {C_1} d(x)^{\frac{\alpha_1 p}{p-1-\alpha_1}}\;\text{ and }\;-\Delta_q\underline{v}\leq  m^{\sigma(q-1)} {C_1'} d(x)^{\frac{\alpha_2q}{q-1-\alpha_2}}.\ee
From \eqref{condexp1} and \eqref{condexp3}, we obtain:
\be\label{subsolu21}
\begin{array}{l}
\lambda_1\underline{u}^{\alpha_1}-\underline{u}^{\beta_1}-\mu_1\underline{u}^{a_1}v^{b_1}\vspace{0.2cm}\\
\geq \lambda_1(m\psi_1)^{\alpha_1}\left[1-\frac{1}{\lambda_1}(m\psi_1)^{\beta_1-\alpha_1}-\frac{\mu_1}{\lambda_1}(m\psi_1)^{a_1-\alpha_1}\left(m^{-\sigma{\rm sign}(b_1)}\psi_2\right)^{b_1}\right]
\vspace{0.2cm}\\
\geq \frac{\lambda_1}{2}m^{\alpha_1}{C_2}d(x)^{\frac{\alpha_1p}{p-1-\alpha_1}},
\end{array}
\ee
for $m$ small enough. In addition, from \eqref{condexp2} and \eqref{condexp4}, we get: 
\be\label{subsolv21}\lambda_2\underline{v}^{\alpha_2}-\underline{v}^{\beta_2}-\mu_2\underline{v}^{a_2}u^{b_2}\geq \frac{\lambda_1}{2}m^{\sigma\alpha_2}{C_2'} d(x)^{\frac{\alpha_2q}{q-1-\alpha_2}}\quad\text{ in }\Omega,\ee
for $m$ small enough. Then, under conditions \eqref{condexp3}, \eqref{condexp4} and for $m$ small enough, $(\underline{u},\underline{v})$ is a subsolutions pair of problem $\boldsymbol{({\rm P})}$. 
We also get
\be\label{supersolu11} -\Delta_p\overline{u}\geq m^{1-p}{C_3}d(x)^{\frac{\alpha_1p}{p-1-\alpha_1}}\text{ and } -\Delta_q\overline{v}\geq m^{\sigma(1-q)}C_3'd(x)^{\frac{\alpha_2q}{q-1-\alpha_2}}\text{ in }\Omega.\ee
Similarly, one has
\be\label{supersolu21} 
\lambda_1 \overline{u}^{\alpha_1}-\overline{u}^{\beta_1}-\mu_1\overline{u}^{a_1}v^{b_1}
\leq
\lambda_1 m^{-\alpha_1}{C_4} d(x)^{\frac{\alpha_1p}{p-1-\alpha_1}}\quad\text{ in }\Omega,
\ee
\be\label{supersolv21}
\lambda_2 \overline{v}^{\alpha_2}-\overline{v}^{\beta_2}-\mu_2\overline{v}^{a_2}u^{b_2}
\leq
\lambda_2 m^{-\sigma\alpha_2}{C_4'}d(x)^{\frac{\alpha_2q}{q-1-\alpha_2}}\quad\text{ in }\Omega.
\ee
Then, for $m $ small enough, $(\overline{u},\overline{v})$ is a supersolutions pair of problem $\boldsymbol{({\rm P})}$. 

Applying Theorem \ref{TH1}, we get the existence of  positive solutions $(u,v)\in{\rm W}^{1,p}_0(\Omega)\times{\rm W}^{1,q}_0(\Omega)$ of $\boldsymbol{({\rm P})}$ satisfying \eqref{estsol1}. From Theorem \ref{regu}, we get the H\"older regularity of $u$ and $v$. This proves (i).

\textit{\textbf{Alternative 2:}} Now, let conditions \eqref{condexp5} and \eqref{condexp6} be satisfied. Then, 
\be\label{estsssol3} \psi_{1}(x)\sim d(x)\quad\text{ and }\quad\psi_{2}(x)\sim d(x)\quad\text{ in }\Omega.\nonumber\ee
Let $(u,v)\in[\underline{u},\overline{u}]\times\left[\underline{v},\overline{v}\right]$.
Instead of \eqref{subsolu11}, we now get
\be -\Delta_p\underline{u}\leq  m^{p-1} {C_1}d(x)^{\alpha_1}\quad\text{ and }\quad -\Delta_q\underline{v}\leq  m^{\sigma(q-1)} {C_1'} d(x)^{\alpha_2}\quad\text{ in }\Omega.\nonumber\ee
From \eqref{condexp1}, \eqref{condexp2}, \eqref{condexp5} and \eqref{condexp6},  instead of \eqref{subsolu21} and \eqref{subsolv21}, we have
\be\lambda_1\underline{u}^{\alpha_1}-\underline{u}^{\beta_1}-\mu_1\underline{u}^{a_1}v^{b_1}\geq \frac{\lambda_1}{2}m^{\alpha_1}{C_2}d(x)^{\alpha_1}\quad\text{ in }\Omega,\nonumber\ee
\be  
\lambda_2\underline{v}^{\alpha_2}-\underline{v}^{\beta_2}-\mu_2\underline{v}^{a_2}u^{b_2}\geq \frac{\lambda_2}{2}m^{\sigma\alpha_2}{C_2'}d(x)^{\alpha_2}\quad\text{ in }\Omega\nonumber,\ee
 for $m$ small enough.
Then, under conditions \eqref{condexp5}, \eqref{condexp6} and for $m $ small enough, $(\underline{u},\underline{v})$ is a subsolutions pair of problem $\boldsymbol{({\rm P})}$. 
 Instead of \eqref{supersolu11}, we have
\be -\Delta_p\overline{u}\geq m^{1-p}{C_3}d(x)^{\alpha_1}\quad\text{ and }\quad -\Delta_q\overline{v}\geq m^{\sigma(1-q)}{C_3'}d(x)^{\alpha_2}\quad\text{ in }\Omega.\nonumber\ee
Furthermore, the following inequalities
\be 
\lambda_1 \overline{u}^{\alpha_1}-\overline{u}^{\beta_1}-\mu_1\overline{u}^{a_1}v^{b_1}
\leq
\lambda_1 m^{-\alpha_1}{C_4} d(x)^{\alpha_1}\;\text{ in }\Omega,
\nonumber\ee
\be
\lambda_2 \overline{v}^{\alpha_2}-\overline{v}^{\beta_2}-\mu_2\overline{v}^{a_2}u^{b_2}
\leq
\lambda_1 m^{-\sigma\alpha_2}{C_4'} d(x)^{\alpha_2}\;\text{ in }\Omega
\nonumber\ee
replace \eqref{supersolu21} and \eqref{supersolv21}.
Then, for $m $ small enough, $(\overline{u},\overline{v})$ is a supersolutions pair of problem $\boldsymbol{({\rm P})}$. 
We conclude as in the {{Alternative 1}} and (ii) is proved.

\textit{\textbf{Alternative 3}}: Now, assume that conditions \eqref{condexp7} and \eqref{condexp8} are satisfied. Then, 
\be\label{estsssol5} \psi_{1}(x)\sim d(x)^{\frac{p}{p-1-\alpha_1}}\quad \text{ and }\quad\psi_{2}(x)\sim d(x)\quad\text{ in }\Omega.\nonumber\ee
Let $(u,v)\in[\underline{u},\overline{u}]\times\left[\underline{v},\overline{v}\right]$.
Instead of \eqref{subsolu11}, we have 
\be -\Delta_p\underline{u}\leq  m^{p-1} {C_1} d(x)^{\frac{\alpha_1p}{p-1-\alpha_1}}\quad\text{ and }\quad  -\Delta_q\underline{v}\leq  m^{\sigma(q-1)} {C_1'} d(x)^{\alpha_2}\quad\text{ in }\Omega.\nonumber\ee
From \eqref{condexp1},  \eqref{condexp2}, \eqref{condexp7} and \eqref{condexp8},  instead of \eqref{subsolu21} and \eqref{subsolv21}, we get
\be\lambda_1\underline{u}^{\alpha_1}-\underline{u}^{\beta_1}-\mu_1\underline{u}^{a_1}v^{b_1}\geq \frac{\lambda_1}{2}m^{\alpha_1}{C_2}d(x)^{\frac{\alpha_1p}{p-1-\alpha_1}}\quad\text{ in }\Omega,\nonumber\ee
\be\lambda_2\underline{v}^{\alpha_2}-\underline{v}^{\beta_2}-\mu_2\underline{v}^{a_2}u^{b_2}\geq \frac{\lambda_2}{2}m^{\sigma\alpha_2}{C_2'}d(x)^{\alpha_2}\quad\text{ in }\Omega,\nonumber\ee
for $m$ small enough.
Then, under conditions \eqref{condexp7}, \eqref{condexp8} and for $m $ small enough, $(\underline{u},\underline{v})$ is a subsolutions pair of problem $\boldsymbol{({\rm P})}$. 
Finally, Instead of \eqref{supersolu11}, we have 
\be -\Delta_p\overline{u}\geq m^{1-p}{C_3}d(x)^{\frac{\alpha_1p}{p-1-\alpha_1}}\quad\text{ and }\quad -\Delta_q\overline{v}\geq m^{\sigma(1-q)}{C_3'}d(x)^{\alpha_2}\quad\text{ in }\Omega.\nonumber\ee
Instead of \eqref{supersolu21} and \eqref{supersolv21}, we obtain
\be 
\lambda_1 \overline{u}^{\alpha_1}-\overline{u}^{\beta_1}-\mu_1\overline{u}^{a_1}v^{b_1}
\leq
\lambda_1 m^{-\alpha_1}{C_4} d(x)^{\frac{\alpha_1p}{p-1-\alpha_1}}\quad\text{ in }\Omega,
\nonumber\ee
\be
\lambda_2 \overline{v}^{\alpha_2}-\overline{v}^{\beta_2}-\mu_2\overline{v}^{a_2}u^{b_2}
\leq
\lambda_1 m^{-\sigma\alpha_2}{C_4'} d(x)^{\alpha_2}\quad\text{ in }\Omega.
\nonumber\ee
Then, for $m$ small enough, $(\overline{u},\overline{v})$ is a supersolutions pair of problem $\boldsymbol{({\rm P})}$. 
Then, we conclude as in the \textit{Alternative 1}. Thus, (iii) and by symmetry (iv) are proved.
\end{proof}
{ Concerning the above theorem, we analyse further some limiting cases. The proof of the next result follows the proof of Theorem \ref{theo2}. So we omit it.}
\begin{theorem}
\begin{enumerate}
\item Let
\be\label{condexpcrit1} \alpha_1=-1\quad \text{ and }\quad (a_1-\alpha_1+b_1)q-(a_1-\alpha_1)(\alpha_2+1)>0,\ee
\be\label{condexpcrit2} -2-\frac{1}{q-1}<\alpha_2<-1\quad\text{ and }\quad (a_2-\alpha_2-b_2)q-b_2(\alpha_2+1)>0.\ee
Then, $\boldsymbol{({\rm P})}$ admits  positive solutions  $(u,v)\in{\rm W}^{1,p}_0(\Omega)\times{\rm W}^{1,q}_0(\Omega)$ satisfying:
\be\label{estsol9} u(x)\sim d(x)\vert \ln(d(x))\vert^{\frac{1}{p}}\quad\text{ and }\quad v(x)\sim d(x)^{\frac{q}{q-1-\alpha_2}}\quad\text{ in }\Omega.\ee
In addition, we have $(u,v)\in\mathscr{C}^{0,\alpha}\left(\overline{\Omega}\right)\times\mathscr{C}^{0,\alpha}\left(\overline{\Omega}\right),$ for some $0<\alpha<1$.
\item Let
\be\label{condexpcrit3} \alpha_1=-1\quad\text{ and }\quad a_1-\alpha_1+b_1>0,\ee
\be\label{condexpcrit4} \alpha_2=-1\quad\text{ and }\quad a_2-\alpha_2+b_2>0.\ee
Then, $\boldsymbol{({\rm P})}$ admits positive solutions  $(u,v)\in{\rm W}^{1,p}_0(\Omega)\times{\rm W}^{1,q}_0(\Omega)$ satisfying:
\be\label{estsol11} u(x)\sim d(x)\vert \ln(d(x))\vert ^{\frac{1}{p}}\;\text{ and }\;v(x)\sim d(x)\vert \ln(d(x))\vert^{\frac{1}{q}}\quad\text{ in }\Omega.\ee
In addition, we have $(u,v)\in\mathscr{C}^{0,\alpha}\left(\overline{\Omega}\right)\times\mathscr{C}^{0,\alpha}\left(\overline{\Omega}\right),$ for some $0<\alpha<1$.
\item Let
\be\label{condexpcrit5} \alpha_1=-1\quad\text{ and }\quad a_1-\alpha_1+b_1>0,\ee
\be\label{condexpcrit6} -1<\alpha_2<q-1\quad\text{ and }\quad a_2-\alpha_2+b_2>0.\ee
Then, $\boldsymbol{({\rm P})}$ admits  positive  solutions  $(u,v)\in{\rm W}^{1,p}_0(\Omega)\times{\rm W}^{1,q}_0(\Omega)$ satisfying:
\be\label{estsol13} u(x)\sim d(x)\vert \ln(d(x))\vert ^{\frac{1}{p}}\quad\text{ and }\quad v(x)\sim d(x)\quad\text{ in }\Omega.\ee
In addition, we have $(u,v)\in\mathscr{C}^{0,\alpha}\left(\overline{\Omega}\right)\times\mathscr{C}^{1,\alpha}\left(\overline{\Omega}\right),$ for some $0<\alpha<1$.
\end{enumerate}
\end{theorem}

\appendix

\section{A useful H\"older regularity result}
We consider the following quasilinear elliptic boundary value problem,
\be\label{bo2} -\Delta_r w=f\quad\text{ in  }\Omega;\quad w\vert_{\partial\Omega}=0,\quad w>0\quad\text{ in }\Omega.\ee
In this equation, $f$ is a ${\rm L}^{1}_{{\rm loc}}(\Omega)$  function such that there exist two constants $C>0$ and $\delta>0$ satisfying 
\be\label{f} \vert f(x)\vert\leq Cd(x)^{-\delta}, \quad\text{ a.e. in }\Omega.\ee
Then, we have the following H\"older regularity result on the solutions to \eqref{bo2}.
\begin{theorem} \label{regu}
Assume that $f$ satisfies the growth hypothesis \eqref{f}. Let $u\in{\rm W}^{1,r}_0(\Omega)$ be a positive weak solution to \eqref{bo2}. Let $\overline{u}\in{\rm W}^{1,r}_0(\Omega)$ be a supersolutions to \eqref{bo2} such that 
\be\label{sursolforme}-\Delta_r\overline{u}\geq \vert f\vert \quad\text{ in }\Omega,\ee
in the sense of distributions in ${\rm W}^{-1,r'}(\Omega)$. In addition, assume that there exists $C'>0 $ such that
\be\label{u}0\leq u\leq\overline{u}\leq C'd(x)^{\delta'}\quad\text{ a.e in }\Omega,\ee
with $0<\delta'<\delta$. Finally, let $\alpha$ be an arbitrary number such that $$0<\alpha<\frac{r}{r-1+\delta/\delta'}<1.$$ Then, there exists a constant $M>0$, depending solely on $\Omega$, $r$ and $N$, on the constants $C$ and $\delta$ in \eqref{f}, on the constants $C'$ and $\delta'$ in \eqref{u}, and on the constant $\alpha$, such that $u\in\mathscr{C}^{0,\alpha}(\overline{\Omega})$ and 
$$\Vert u\Vert_{\mathscr{C}^{0,\alpha}(\overline{\Omega})}\leq M.$$
\end{theorem}
\begin{proof}
The proof is quite similar to the Theorem 1.1's in \cite{GST1} with $\boldsymbol{{\rm a}}:(x,\eta)\mapsto \vert\eta\vert^{p-2}\eta$ in $\Omega\times\RR^N$.
Indeed, to overcome the non-positivity of $f$, we add conditions \eqref{sursolforme} and \eqref{u}. Then, introducing the same boundary value problem (2.12),instead of inequality (2.14), we get here
\be \vert u(x)-v(x)\vert \leq \overline{u}(x) \leq C x_N^{\delta'}\quad\text{ for all } x=(x',x_N)\in B_R^+(0).\ee
Then, estimate (A.18) still holds and the end of the proof is exactly the same.
~
\end{proof}

\end{document}